\documentclass[11pt,letterpaper]{article}

\usepackage{amssymb}
\usepackage{amsthm}
\usepackage{amsmath}
\usepackage[pdftex]{graphicx}
\usepackage{graphicx}
\usepackage{fullpage}
\usepackage[margin=1in]{geometry}
\usepackage{enumerate}
\usepackage{algorithm,algorithmic}

\usepackage[colorlinks=red,linkcolor=red]{hyperref}
\makeatletter
\renewcommand*{\eqref}[1]{%
  \hyperref[{#1}]{\textup{\tagform@{\ref*{#1}}}}%
}
\makeatother

\usepackage{color}

\newtheorem{theorem}{Theorem}[section]

\newtheorem{proposition}[theorem]{Proposition}

\theoremstyle{definition}
\newtheorem*{definition}{Definition}

\theoremstyle{remark}

\title{Dictionary Learning under Symmetries via Group Representations} 
\author{Subhroshekhar Ghosh$^1$, Aaron Y. R. Low$^1$, Yong Sheng Soh$^{1}$, \\
 Zhuohang Feng$^2$, and Brendan K. Y. Tan$^1$ \\ \\
$^1$Department of Mathematics, National University of Singapore \\ 21 Lower Kent Ridge Road, Singapore 119077 \\ \\
$^2$School of Mathematical Sciences, University of Chinese Academy of Sciences \\ 19A Yuquan Road, Beijing 100049, P. R. China}

\begin{document}

\maketitle

\begin{abstract}
The dictionary learning problem can be viewed as a data-driven process to learn a suitable transformation so that data is sparsely represented directly from example data. In this paper, we examine the problem of learning a dictionary that is invariant under a pre-specified group of transformations. Natural settings include Cryo-EM, multi-object tracking, synchronization, pose estimation, etc. We specifically study this problem under the lens of mathematical representation theory. Leveraging the power of non-abelian Fourier analysis for functions over compact groups, we prescribe an algorithmic recipe for learning dictionaries that obey such invariances. We relate the dictionary learning problem in the physical domain, which is naturally modelled as being infinite dimensional, with the associated computational problem, which is necessarily finite dimensional. We establish that the dictionary learning problem can be effectively understood as an optimization instance over certain matrix orbitopes having a particular block-diagonal structure governed by the irreducible representations of the group of symmetries. This perspective enables us to introduce a band-limiting procedure which obtains dimensionality reduction in applications. We provide guarantees for our computational ansatz to provide a desirable dictionary learning outcome.  We apply our paradigm to investigate the dictionary learning problem for the groups SO(2) and SO(3). While the SO(2)-orbitope admits an exact spectrahedral description, substantially less is understood about the SO(3)-orbitope. We describe a tractable spectrahedral outer approximation of the SO(3)-orbitope, and contribute an alternating minimization paradigm to perform optimization in this setting. We provide numerical experiments to highlight the efficacy of our approach in learning SO(3)-invariant dictionaries, both on synthetic and on real world data.

\vspace{0.025in}

\end{abstract}

\section{Introduction} \label{sec:intro}


The ability to model data as sparse vectors under suitable transformations is an essential component of numerous approaches for performing data processing tasks in a range of applications \cite{CDS:98,DonHuo:01}.  Prominent examples of such tasks include signal denoising, imputing missing entries in data (such as in low rank matrix or tensor completion), and solving ill-posed inverse problems \cite{CRT:06,CT:06,Don:95,Don:06}.  At its algorithmic core, these techniques rely on recovering a sparse solution given partial measurements.


\subsection{Dictionary Learning}  Because of the important role that sparsity plays in these approaches, the success of these methods effectively rely on having access to a suitable transformation so that the data of interest is well-modelled as sparse vectors.  The traditional perspective to identifying such transformations is via deep domain expertise -- for instance, the use of the discrete cosine transform and wavelets in image processing follows a long lineage of prior works.  In recent decades, we have witnessed the emergence of an alternative approach in which we \emph{learn} the appropriate transformation \emph{directly} from example data \cite{OlsFie:96,OlsFie:97}.  In the process of doing so, we are now able to readily apply previously developed techniques for data processing applications using these learned  transformations -- this is especially useful in application domains where a suitable transformation is not \emph{a priori} available \cite{Ela:10,OlsFie:96,OlsFie:97}.

The process of learning a suitable transform is referred to as the \emph{dictionary learning} problem in the literature.  Mathematically, given a collection of vectors $\{ y^{(i)} \}_{i=1}^{n} \subset \mathbb{R}^{d}$, we seek a small collection of vectors $\{ \phi_j \}_{j=1}^{q} \subset \mathbb{R}^{d}$ so that every data point is well-approximated by a linear sum of a \emph{few} of these $\phi$'s
\begin{equation} \label{eq:dictlearning_regular}
y^{(i)} \approx \sum_{j=1}^{q} x_j^{(i)} \phi_j, \qquad \text{ where }\qquad x_j^{(i)} \in \mathbb{R}, ~ x^{(i)} = (x_1^{(i)},\ldots,x_q^{(i)} )^{T} \text{ is sparse}.
\end{equation}
The learned $\phi$'s are often referred to as \emph{dictionary atoms}.  A more compact way to express \eqref{eq:dictlearning_regular} is as follows
\begin{equation*} 
y^{(i)} \approx \Phi x^{(i)} \quad \text{ where } \quad x^{(i)} \text{ is sparse},  \text{ and } \quad \Phi = [\phi_1 | \ldots | \phi_q ].
\end{equation*}

The computational task of obtaining the dictionary atoms $\{ \phi_j \}$ is typically formulated as a joint optimization instance over the atoms and the coefficients $x^{(i)}$.  One such formulation is to incorporate the L1-norm to encourage sparsity in the solution $x^{(i)}$
\begin{equation} \label{eq:dl}
\underset{\Phi, x^{(i)} }{\arg \min} ~~ \frac{1}{n} \sum_{i=1}^{n} \left( \| y^{(i)} - \Phi x^{(i)} \|_2^2 + \lambda \| x^{(i)} \|_1 \right) \quad \text{ s.t. } \quad \| \phi_j \|_2^2 = 1.
\end{equation}
Here, $\lambda \geq 0$ is a regularization parameter.  While the optimization instance is \emph{not} convex, the sub-problem corresponding to fixing one set of variables (either the $\phi$'s or the $x$'s) and optimizing over the other is a \emph{convex} program \footnote{To be precise, the sub-problem we refer to when updating $\phi$ is one where we treat $x$ in \eqref{eq:dl} as fixed and we update $\phi$ disregarding the $\ell_2$-norm constraint.  In this case, the update leads to a least squares problem.  The atoms are subsequently scaled to be unit-norm.  This is the usual procedure applied in practice.}.


\subsection{Convolutional Dictionary Learning}  In certain application domains where dictionary learning techniques are applied, it is sometimes more helpful to permit suitably \emph{shifted} or \emph{transformed} copies of a dictionary element as alternative dictionary elements
\begin{equation*}
(a_1,a_2,\ldots, a_d) \quad \mapsto \quad (a_2,a_3,\ldots, a_d, a_1).
\end{equation*}
The \emph{Convolutional Dictionary Learning} (\emph{abbrv.} CDL) problem concerns learning dictionaries with the property that all cyclic shifts (in the coordinates) of a dictionary atom also appear as dictionary atoms.  Invariance of the collection of dictionary atoms under shift transformations is a natural modeling assumption in audio processing (in which we perform shifts in the temporal direction) and in image processing (in which shifts occur in the spatial domain) -- a substantial body of work devoted to developing computational techniques \cite{GaCWoh:18,JVLG:06,Wohlberg:15,ZSE:19} and applying these to audio processing tasks \cite{BluDav:06,GRKN:07,LewSej:17,PABD:06}, image processing tasks \cite{LGWY:18,PRSE:17}, and time-series analysis \cite{Rusu:20,SFB:19} have since followed.  

The CDL computational task can be phrased, for instance, in terms of the following minimization instance
\begin{equation} \label{eq:CDL_orig}
\underset{\Phi, x^{(i)} }{\arg \min} ~~ \frac{1}{n} \sum_{i=1}^{n} \left( \| y^{(i)} - \sum_{j=1}^{q} \phi_j \ast x^{(i)} \|_2^2 + \lambda \| x^{(i)} \|_1 \right) \quad \text{ s.t. } \quad \| \phi_j \|_2^2 = 1,
\end{equation}
where $\ast$ denotes convolution, as per usual convention.

If we are to express the problem explicitly in the dictionary atoms, then the formulation would be exactly as in \eqref{eq:dl}, with the modification whereby $\Phi$ has the following structure
\begin{equation} \label{eq:CDL_expanded}
\Phi  = 
\left[ \begin{array}{c|c|c|c|c}
 \Phi_0 & T \, \Phi_0 & T^2 \, \Phi_0 & \ldots & T^{d-1} \, \Phi_0 
\end{array}\right] \quad \text{ where } \quad \Phi_0 = \left[ \begin{array}{c|c|c}
 \phi_1 & \ldots & \phi_q 
\end{array}\right].
\end{equation}
Here, $T^k$ denotes a shift operation by $k$ coordinates.  In this language, $\Phi$ represents the full set of dictionary atoms while $\Phi_0$ represents the set of generators.

\subsection{Group Invariant Dictionary Learning}  This raises the natural question: how do we handle more general symmetries beyond shifts? Applications where there is an action of a natural group of symmetries are many. Prominent examples include natural image data under rotations and translations \cite{Bispectrum:12,Kon07}; data arising in tomography \cite{Mor12,Vas07}; network based data \cite{Gir06,Her18}; synchronization problems in robotics \cite{Ros20}; computer vision \cite{Agr06}; multi-reference alignment (MRA) \cite{BBKNPW:23,bandeira2014multireference, FLSWX:24,fan2020likelihood, ghosh2022sparse,perry2019sample} and cryo electron microscopy (cryo EM) under the action of rigid motions \cite{BBKNPW:23,bendory2020single,FLSWX:24,singer2018mathematics}.

If the underlying group of transformations is finite then the solution is immediate -- we simply append our dictionary with all transformed copies of the atoms as we did in CDL.  However, this approach is impractical for large groups.  Moreover, if the symmetry is \emph{continuous}, then such a strategy would not work, even at a conceptual level. 

To overcome difficulties arising from working with continuous symmetries, Soh \cite{Soh:20} proposed a different alternating minimization framework for learning dictionaries that are invariant to some specified group invariance.  Concretely, this task, which we term  as the \textit{group invariant dictionary learning} (\emph{abbrv.} GIDL) problem, can be described as learning a collection of atoms $\Phi_0 :=\{ \phi_j \}_{j=1}^{q}$ satisfying the following:
\begin{equation} \label{eq:gidl_problem}
y^{(i)} \approx \sum_{1 \leq k \leq K} x_k^{(i)} \,  g_k^{(i)} \cdot \phi_{[k]}, \qquad x_k^{(i)} \in \mathbb{R}, g_k^{(i)} \in G, \phi_{[k]} \in \Phi_0, \qquad 1 \leq i \leq n.
\end{equation}
That is to say, each data-point $y$ is well-approximated as linear sums of terms of dictionary atoms of the form $g_k \cdot \phi_{[k]}$, each of which are obtained by applying a symmetry transformation $g_k \in G$ on the generating atom $\phi_{[k]}$.  Here, we use square brackets in the subscript to emphasize that $\phi_{[k]}$ is a choice of generating atom from $\Phi_0$, and has no relation to $\phi_k$.  In particular, the model allows for a generating atom $\phi_j$ to participate in \eqref{eq:gidl_problem} more than once, or possibly never at all.  In addition, we impose the additional constraint that the cardinality of the index set should be small relative to the data dimension; i.e. $K \ll d$, so as to capture model sparsity.

Suppose, for now, that the symmetry $g$ can be expressed as the action of a linear transformation $\rho(g)$ acting on a vector space $V$
\begin{equation*}
g \cdot \phi = \rho (g) \phi \quad \text{ for all } \quad g \in G, \phi \in V.
\end{equation*}
The key contribution in \cite{Soh:20} is to formulate the GIDL problem as the solution of the following optimization instance
\begin{equation} \label{eq:GIDL_opt2}
\underset{L_j^{(i)}, \phi_j}{\arg\min}~~\sum_i \big( \frac{1}{2} \| y^{(i)} - \sum_{j=1}^{q} L^{(i)}_{j} \phi_j \|_2^2 + \lambda \cdot \sum_{j=1}^{q} \| L^{(i)}_{j} \|_G \big) \qquad \text{s.t.} \qquad \| \phi_j \|_2^2 = 1,
\end{equation}
where the optimization variables $L_j^{(i)}$ are \emph{operators} residing in the space $\mathrm{Span}( \{ \rho (g) : g \in G \})$.  Here, $\| \cdot \|_G$ is a convex penalty function -- specifically, it is the Minkowski functional induced by the convex hull of the $\rho$'s \cite{BTR:13,CRPW:12}
\begin{equation*} \label{eq:atomic_norm}
\| L \|_G : = \inf \{ ~ t : L \in t \cdot \mathrm{conv}( \{ \pm \rho(g) : g \in G \} ) ~ \}.
\end{equation*}
In particular, the $\| \cdot \|_G$ has the effect of promoting solutions $L^{(i)}_{j}$ that are sparse linear sums of elements from the set $\{ \rho(g) : g \in G \}$.  

The conceptual difference \eqref{eq:GIDL_opt2} makes from \eqref{eq:CDL_orig} and \eqref{eq:CDL_expanded} is that it lifts the optimization problem from the vector space where the data resides to the vector space where the $\rho$'s representing the transformation reside.  In doing so, the problem of representing a data vector $y \in \mathbb{R}^d$ sparsely with respect to a group invariant dictionary -- even if it is infinite or continuous -- can be resolved by solving a \emph{finite} dimensional convex program; specifically by minimizing with respect to the penalty function $\| \cdot \|_G$. 

It is perhaps useful to clarify some conceptual aspects of our framework: the manner in which we represent data in \eqref{eq:GIDL_opt2} -- namely $y \approx \sum_{j=1}^{q} L_{j} \phi_j$ -- in principle provides us with a very large number of parameters to express the data-point $y$.  It is, however, through the presence of the penalty function $\| \cdot \|_G$ applied onto the $L_{j}$'s that we obtain sparsity within the $L_{j}$'s, and from which we learn succinct representations of the data $y$.

\subsection{Our contributions}

The framework laid out by Soh in \cite{Soh:20} fundamentally relies on two basic ingredients being satisfied for its success
\begin{enumerate}
\item We have an explicit description of $\rho$; and \label{requirement 1}
\item We are able to express the convex hull of $\{ \rho(g) : g \in G \}$ tractably.
\end{enumerate}
In what follows, we lay out our four main contributions in this paper in the above context.

Our {\bf first contribution} in this work is to address requirement \ref{requirement 1}.  Our work focuses on settings where the invariances can be described via the action of a \emph{compact} group $G$ whereby we provide a systematic recipe for expressing the transformation by a certain explicit action of linear operators.  We present key features of this programme below; for details we refer the reader to Sections \ref{sec_preliminaries} and \ref{sec_framework}.  The first major step is to represent the data as \emph{functions} over the group $G$ -- this is an essential modeling step, without which it is not apparent how one systematically produces descriptions of $\rho$ that is amenable to computation.  Second, a fundamental result in the representation theory of compact groups, namely the Peter-Weyl Theorem, entails that every such function admits a generalized Fourier series -- unlike the regular Fourier series over periodic functions we are familiar with, the Fourier coefficients may be \emph{matrix}-valued (as we shall see later, this happens if the group is non-commutative).  Third, the Peter-Weyl theory also entails that the group action (namely, the \emph{regular representation}) can be expressed as an operator with the same block diagonal structure as the Fourier series.  That is to say, the group action decomposes to a sequence of square blocks, and each block only acts on a single matrix-valued Fourier coefficient, via an explicit matrix multiplication operation.


An important objective of this paper is to spell out the consequences of group representation theory in our dictionary learning problem in a language that is readily accessible to practitioners.  On the computational front, there are still a number of issues to resolve.  For continuous groups $G$, the functions and operators both reside in infinite dimensional spaces -- what then is the appropriate finite dimensional approximation of the dictionary learning problem?  Our {\bf second contribution} is to address these problems.

In this vein, we explore in greater depth the harmonic analytical aspects of our dictionary learning problem.  In particular, by drawing the connections of representation theory to our group invariant dictionary learning problem, we are now in a position to pose the following question:
\begin{quote}
The dictionary learning problem we solve is a computational problem that resides over a finite dimensional space.  In reality, the problem we solve is simply a model of a similar task that resides in the physical domain, which is continuous.  Is there a principled way of describing what the dictionary learning problem is in the continuous limit?  In particular, are we able to also quantify the error between the dictionary elements we learned by solving the finite dimensional approximation from the analogous problem in the continuous limit?
\end{quote}

Our {\bf third contribution} focuses on the harmonic analytical aspects of our dictionary learning problem.  Under suitable qualifications, we show that the dictionary elements learned in the finite dimensional approximation converges to the solution of the actual problem in the limit, as the dimensionality increases.

Our {\bf fourth contribution} is to spell out in detail the above programme for concrete examples of certain groups that are of key importance in applications.  On this note, we first discuss the process for SO(2), and how this problem can in fact be viewed as the continuous limit of the CDL problem.  Second and third, we discuss the agenda for O(2) and SO(3) respectively.  The discussion for SO(3) is substantially more difficult as it is unknown whether the convex hull of the dictionary elements admits a tractable description.  The apparent connection to the tensor nuclear norm ball suggests that this set is probably not tractable to describe.  Nevertheless, one can easily write down a simple convex relaxation, and our numerics suggest that the relaxation is actually quite tight.  

\section{Preliminaries} \label{sec_preliminaries}

In this section, we introduce the necessary concepts from group representation theory, harmonic analysis and convex geometry.  In particular, we lay out the groundwork for analysing a non-abelian Fourier series expansion for functions over compact groups in the context of the GIDL problem.  It is on these non-abelian, matrix-valued Fourier coefficients that our principal dictionary learning algorithms operate.

\subsection{Preliminaries on Group Theory}
  
Let $G$ be a multiplicative group whose identity element is denoted as $e$; that is, $a a^{-1} = a^{-1} a = e$ $\forall a \in G$.  A \emph{topological} group is a group that is also a topological space such that multiplication $(a, b) \mapsto ab$ as well as the inverse $a \mapsto a^{-1}$ are both continuous.  A \emph{compact} topological group is a topological group whose topology is compact.  (Throughout this paper, we assume that the topology is Hausdorff.)

In what follows, we use $L^2(G)$ to denote the Hilbert space of squared integrable functions $f:G \rightarrow \mathbb{C}$; that is,
\begin{equation*}
f \in L^2(G) \qquad \Leftrightarrow \qquad \| f \|^2_2 := \| f \|_{L^2(G)}^2 = \int_{G} |f(g)|^2 d \mu (g) <  \infty.
\end{equation*}
Here, $\mu$ always denotes the normalized uniform (Haar) measure over the group.  In this paper, we focus exclusively on compact groups, and hence the uniform measure exists.  

In addition, we are particularly interested in operators on $L^2(G)$.  Concretely, given a Hilbert space $H$ (very often, this will be $L^2(G)$), we let $B(H)$ denote the normed vector space of bounded linear operators $H \rightarrow H$.  Similarly, we let $U(H)$ denote the subset of unitary operators.

\textbf{Group Representations.}  Given two groups $G$ and $H$, a homomorphism from $G$ to $H$ is a map $\rho : G \rightarrow H$ such that 
\begin{equation*}
\rho(ab) = \rho(a) \rho(b) \qquad \text{ for all } \qquad a,b \in G.
\end{equation*}

Let $G$ be a group and let $V$ be a vector space over $\mathbb{C}$.  A \emph{representation} of $G$ on $V$ is a (continuous) homomorphism from $G$ to $\mathrm{GL}(V)$, the group of invertible matrices on $V$.  The \emph{dimension} of a representation $\rho: G \rightarrow \mathrm{GL}(V)$ is the dimension of the vector space $V$.  In what follows, we sometimes use $V$ to refer to the representation.

We say that two representations $\rho_1: G \rightarrow \mathrm{GL}(V_1)$ and $\rho_2: G \rightarrow \mathrm{GL}(V_2)$ are \emph{isomorphic} if there exists an isomorphism $\pi : V_1 \rightarrow V_2$ such that 
\begin{equation*}
\pi \rho_1 (g) \pi^{-1} = \rho_2 (g) \qquad \text{ for all } \qquad g \in G.
\end{equation*}

Let $\rho$ be a representation of $G$ on $V$.  A subspace $W \subseteq V$ that is invariant under the group action $G$ -- that is, it satisfies $\rho(g) w \in W$ for all $g \in G$ and all $w \in W$ -- is called a \emph{subrepresentation}.  We say that a representation $\rho$ is \emph{irreducible} if the only subrepresentations of $V$ are the zero vector space and $V$.

\textbf{Non-commutative Fourier Expansion for Compact Groups.} In the remainder of this paper, we assume that $G$ is a compact (topological) group.  We proceed to describe the main result concerning the existence of a Fourier series for functions over $G$.

First, all the irreducible representations of a compact group are \emph{finite} dimensional.  In fact, if the group is abelian, then the irreducible representations are all \emph{one}-dimensional.

Second, a \emph{unitary representation} is a (continuous) homomorphism $\rho : G \rightarrow U(H)$ for some Hilbert space $H$, where $U(H)$ denotes the group of unitary operators acting on $H$.  In other words, a unitary representation sends every element $g \in G$ to a unitary operator.  It is possible to show that, given any irreducible representation of a compact $G$, one can apply an appropriate change of basis so that the irreducible representation is unitary.  This choice is particularly convenient and is standard.  From henceforth, we assume that irreducible representations are always chosen to be unitary.  Finally, the most important example of a representation in our set-up is the \emph{left regular representation} $\tau : G \rightarrow U(L^2(G))$ acting on $L^2(G)$ as follows
\begin{equation} \label{eq:leftregularaction}
[\tau(g) f](x) := f(g^{-1} x).
\end{equation}

Third, given a compact $G$, one can enumerate all the irreducible finite dimensional unitary representations $\rho_{\xi} : G \rightarrow U(V_\xi)$ of the group $G$, up to isomorphism. (We will assume in this paper that $G$ is second countable for the previous statement to hold.) We let $\hat{G}$ denote such an enumeration.  For convenience, we assume that $\hat{G}$ is always enumerated according to the dimension of the irreducible representation.  

Fourth, the Peter-Weyl Theorem tells us that the regular representation $\tau : G \rightarrow U(L^2(G))$ is isomorphic (via unitaries) to  a direct sum of all irreducible representations given as follows
\begin{equation} \label{eq:tau_decomposition}
\tau \cong \underset{\xi \in \hat{G}}{\oplus} \rho_\xi^{\oplus \mathrm{dim}(V_\xi)}.
\end{equation}

Consider a function $f \in L^2(G)$.  We define the Fourier coefficient of $f$ at $\xi \in \hat{G}$ as follows
\begin{equation}
\hat{f}_\xi = \hat{f}(\xi) := \int_{G} f(g) \rho_\xi (g)^* ~d\mu (g).
\end{equation}
Here, note that $\rho_\xi (g)$ is a square complex matrix of dimensions $\mathrm{dim}(V_\xi) \times \mathrm{dim}(V_\xi)$.  For $\xi \in \hat{G}$, we define the usual trace inner product over the space of linear transformations from $V_\xi$ to $V_\xi$
\begin{equation}
\langle A, B \rangle_\xi := \langle A, B \rangle_{HS(V_\xi)} = \mathrm{tr}_{V_\xi} (AB^*).
\end{equation}
We then have the following Fourier series expansion
\begin{equation} \label{eq:noncomm_fourier}
f(x) ~=~ \sum_{\xi \in \hat{G}} \mathrm{dim}(V_\xi) ~ \left \langle \hat{f}(\xi), \rho_\xi(x)^* \right\rangle_\xi ~=~ \sum_{\xi \in \hat{G}} \mathrm{dim}(V_\xi) ~ \mathrm{tr}_{V_\xi}( \hat{f}(\xi) \rho_\xi(x) ) .
\end{equation}

Let us examine what the left regular representation \eqref{eq:leftregularaction} does on the Fourier coefficients.  Using \eqref{eq:noncomm_fourier}, we have
\begin{equation}
[\tau(g) f](x) ~ = ~ \sum_{\xi \in \hat{G}} \mathrm{dim}(V_\xi) ~ \mathrm{tr}( \hat{f}(\xi) \rho_\xi(g)^* \rho_\xi(x) ) .
\end{equation}
In other words, when we express a function $f\in L^2(G)$ in the Fourier series, the left regular representation $\tau(g)$ acts on the Fourier coefficients by right multiplication of matrices:
\begin{equation*}
\tau(g) : \hat{f}(\xi)  \mapsto \hat{f}(\xi) \rho_\xi(g)^*.
\end{equation*}

Actually, this relation extends a bit more.  Consider the following bi-action
\begin{equation*}
[\tau(g,h) f] (x) ~:=~ f(g^{-1} x h).
\end{equation*}
Then the bi-action $\tau(g,h)$ acts on the Fourier coefficients by conjugation 
\begin{equation} \label{eq:biaction}
\tau(g,h) : \hat{f}(\xi)  \mapsto \rho_\xi (h) \hat{f}(\xi) \rho_\xi(g)^*.
\end{equation}

In addition, we also see that the left regular representation respects the block structure in that an action on a Fourier coefficient does not affect the other Fourier coefficients.  In what follows, we say that an operator $t \in B(L^2(G))$ has a block diagonal structure (according to $\hat{G}$) if
\begin{equation*}
t ( L^2(G)_{\xi}) \subseteq L^2(G)_{\xi}.
\end{equation*}
Here, $L^2(G)_{\xi}$ denotes the $\xi$-isotypic component
of $L^2(G)$; i.e., it is the sub-representation corresponding to $\rho_{\xi}^{\oplus \mathrm{dim}(V_{\xi})}$ in \eqref{eq:tau_decomposition}.  Finally, the \emph{Plancheral identity} gives us the following 
\begin{equation} \label{eq:plancheral}
\| f \|^2 = \sum_{\xi \in \hat{G}} \mathrm{dim}(V_\xi)^2 \| \hat{f} (\xi) \|^2.
\end{equation}
(Note, the norm $\| \hat{f} (\xi) \| = \| \hat{f} (\xi) \|_F$ is precisely the Frobenius norm induced by the trace inner product.)

\subsection{Preliminaries on Convex Geometry}

Let $\mathcal{C}$ be a centrally symmetric set (that is, $-x \in \mathcal{C}$ if and only if $x \in \mathcal{C}$).  The \emph{atomic norm} with respect to $\mathcal{C}$ is the \emph{Minkowski functional} induced by closed convex hull of $\mathcal{C}$ (denoted $\mathrm{cl}(\mathrm{conv}(\mathcal{C}))$) is defined as follows :
\begin{equation} \label{eq:atomicnorm_defn}
\| x \|_{\mathcal{C}} ~ := ~ \inf \{ ~ t > 0 : x \in t \cdot \mathrm{cl}(\mathrm{conv}(\mathcal{C})) ~ \}.
\end{equation}
In particular, $\| \cdot \|_{\mathcal{C}}$ is a positively homogenous function whose level set is the closed convex hull of $\mathcal{C}$.  By convention, we say that $\| x \|_{\mathcal{C}} = +\infty$ if there does not exist $t > 0$ for which $x \in t \cdot \mathrm{cl}(\mathrm{conv}(\mathcal{C})) $.

In this paper, the atomic norm we primarily use is the one induced by the left regular representations of $G$. Let $\tau$ be the canonical left regular representation of $G$, and let $\lambda(G)$ denote the collection of all (signed) linear operators  (on $L^2(G)$) that are obtained from $\tau$ 
\begin{equation} \label{eq:atomicnorm_unitaries}
\lambda(G) ~ := ~ \{ \pm \tau(g) : g \in G \} ~~ \subseteq ~~ U(L^2(G)).
\end{equation}
Consider the normed vector space of operators $B(L^2(G))$.  We define $\| \cdot \|_{G}:= \| \cdot \|_{\lambda(G)}$ over $B(L^2(G))$ as the atomic norm induced by $\lambda(G)$.  

In the definition of set $\lambda(G)$ in \eqref{eq:atomicnorm_unitaries}, we include negations because our signal model \eqref{eq:gidl_problem} permits negative coefficients.  Correspondingly, if our signal model \eqref{eq:gidl_problem} permits complex coefficients, we would then extend signs in $\tau(g)$ to all phases $S^1$.

Earlier, we made the remark that the regular representations respect the block diagonal structure of the irreducible representations.  The space of such block diagonal operators is a subspace of $B(L^2(G))$, the space of bounded linear operators on $L^2(G)$.  In particular, the atomic norm $\| \cdot \|_{G}$ evaluates to $+ \infty$ if the input operator does not respect the relevant block-diagonal structure.  Therefore, penalizing with the norm $\| \cdot \|_{G}$ ensures that the resulting optimal operator has the desirable block-diagonal structure.

\section{Framework and Algorithm} \label{sec_framework}

\subsection{Problem Statement}

For concreteness, we consider the following problem in this paper.  Let $G$ be a compact group.  Suppose we are provided with a collection of (real-valued) functions over the group $G$
\begin{equation} \label{eq:datamodel_functions}
\{ y^{(i)} \}_{i=1}^{n} \subset L^2(G), \qquad y^{(i)} : G \rightarrow \mathbb{R}.
\end{equation}

The collection of functions $\{ y^{(i)} \}_{i=1}^{n}$ models our dataset.  Our goal is to compute a collection of 
functions $\Phi_0 := \{ \phi_j \}_{j=1}^{q} \subset L^2(G)$, $\phi_j : G \rightarrow \mathbb{R}$, so that every function $y^{(i)}$ is well approximated as linear sums of a few $G$-shifted functions from $\Phi_0$.
Following the introduction of the left regular representation in \eqref{eq:leftregularaction}, the problem \eqref{eq:gidl_problem} may be equivalently expressed as
\begin{equation} \label{eq:gidl_problem2}
y^{(i)} \approx \sum_{1 \leq k \leq K} c_{k}^{(i)} \, \tau( g_{k}^{(i)} ) \phi_{[k]} , \qquad  c_{k}^{(i)} \in \mathbb{R}, g_{k}^{(i)} \in G, \phi_{[k]} \in \Phi_0, \qquad 1 \leq i \leq n.
\end{equation}
By few in \eqref{eq:gidl_problem2} as well as in \eqref{eq:gidl_problem}, the number of atoms $K$ required to approximate each data-point $y^{(i)}$  should be small respect to the problem dimensions $d$ and $n$.  In a similar fashion to \eqref{eq:gidl_problem}, $\phi_{[k]}$ denotes a choice of generating atom from the set $\Phi_0$, and is not be confused with $\phi_{k}$.  In particular, the formulation \eqref{eq:gidl_problem2} allows for a single generating atom $\phi_j$ to appear multiple times and acted on by different $G$-shifts, or never at all, in representing a single data-point $y$.

\subsection{Lifting of Data from Homogenous Spaces to Groups}  \label{sec:lifts}
  
On the surface, our choice to model data as \emph{functions} over a group as in \eqref{eq:datamodel_functions} might seem restrictive, especially if you are a data analyst.  In many applications, it is more typical to view data as residing in some space on which the group transformation $G$ \emph{acts}.  As it turns out, this is not a major obstacle as it is possible to `lift' data as functions over $G$ under very general conditions.  

Concretely, let $\mathcal{X}$ be the space on which the data resides.  We consider a data vector $y \in \mathbb{R}^{\mathcal{X}}$ (in the case where $\mathcal{X}$ is discrete), or more generally, a function $y : \mathcal{X} \rightarrow \mathbb{R}$. In the discrete case, we can also view $y$ as a function $\tilde{y} : \mathcal{X} \mapsto \mathbb{R}$ where
\begin{equation*}
\tilde{y}(x) = y_x, \qquad x \in \mathcal{X}.
\end{equation*}
As an example, if the data are planar images, then one can view an image as a function where the input is a planar coordinate and the output is the image intensity at that coordinate.  

Next, pick an arbitrary $x_0 \in \mathcal{X}$ to be our `origin'.  We define the function $f_y : G \rightarrow \mathbb{R}$ by
\begin{equation} \label{eq:homogenousspace_lift}
f_y (g) = \tilde{y} (g x_0) \quad \text{ for all }\quad g \in G.
\end{equation}
Suppose that the group $G$ acts \emph{transitively} over $\mathcal{X}$; that is, for every $x_1,x_2 \in \mathcal{X}$, there is a $g \in G$ so that $g \cdot x_1 = x_2$ (such a $\mathcal{X}$ is known as a \textit{homogeneous space}).  
Then there is a one-to-one correspondence between $f_z$ and $z$.  
The choice of the origin $x_0 \in \mathcal{X}$ is not essential -- the resulting lifted functions $\{ f_{y^{(i)}} \}$ are equivalent up to a common $G$-shift.

\subsection{Dictionary Learning}

To learn the $\Phi$, we solve the following optimization instance
\begin{equation} \label{eq:GIDL_opt}
\underset{\ell_j^{(i)},\phi_j}{\arg\min}~~\sum_i \big( \frac{1}{2} \| y^{(i)} - \sum_{j=1}^{q} \ell^{(i)}_{j} \phi_j \|_2^2 + \lambda \cdot \sum_{j=1}^{q} \| \ell^{(i)}_{j} \|_G \big) \qquad \text{s.t.} \qquad \| \phi_j \|_2^2 = 1.
\end{equation}
Here, $\phi_j \in L^2(G)$ are functions while $\ell_j^{(i)} \in B (L^2(G))$ are operators.  The penalty function $\| \cdot \|_G$ has the effect of promoting solutions $\ell^{(i)}_{j}$ that are sparse linear sums of elements from the set $\{ \pm \tau(g) : g \in G \}$. Further, the fact that our penalty term is structured as an $\ell_1$ linear combination of the$ \| \cdot \|_G$ norms also ensures sparsity in the set of $\ell^{(i)}_{j}$-s.   We constrain the dictionary elements $\| \phi_j \|_2^2 = 1$, as the function $\| \cdot \|_G$ would otherwise penalize the dictionary elements $\phi_j$ to be zero.

The variables $\ell$ (we omit the dependency on $i$ and $j$) in our optimization instance \eqref{eq:GIDL_opt} are operators, and reside in a much higher dimensional space compared to the functions $\phi_j$.  However, the Peter-Weyl theory tells us that all $\tau(g)$-s have a particular block diagonal structure indexed by $\hat{G}$.  Since the set of block diagonal operators (with the same block structure) is a subspace, it follows that the linear combinations of these representations as they appear in \eqref{eq:gidl_problem2} are also block diagonal.  In other words, we may assume without loss of generality that the operators $ \ell$ in \eqref{eq:GIDL_opt} also possess the same block diagonal structure.  That is to say, the operator $\ell$ is completely specified by a sequence of finite dimensional matrices indexed by $\hat{G}$
\begin{equation*}
\{ \ell_\xi \}_{\xi \in \hat{G}}, \qquad \ell_\xi \in \mathbb{C}^{\dim (V_\xi) \times \dim (V_\xi)}, 
\end{equation*}
and its action on a function $\phi \in L^2(G)$ is described by the action on the individual Fourier coefficients by right multiplication of matrices :
\begin{equation*}
\ell_\xi : \hat{\phi}_{\xi} \mapsto \hat{\phi}_{\xi} \ell_{\xi}^{*}. 
\end{equation*}

We spell out what this observation means in the context of the objective \eqref{eq:GIDL_opt}.  By applying the Plancherel identity \eqref{eq:plancheral}, the objective in \eqref{eq:GIDL_opt} (for each fixed $i$) can be re-written as follows
\begin{equation*}
\frac{1}{2} \| y^{(i)} - \sum_{j=1}^{q} \ell_{j}^{(i)} \phi_j \|_2^2 + \lambda \cdot \sum_{j=1}^{q} \| \ell_{j}^{(i)} \|_G ~=~ \sum_{\xi \in \hat{G}} \frac{\mathrm{dim}(V_\xi)^2}{2} \| \hat{y}_\xi
^{(i)} - \sum_{j=1}^{q} \hat{\phi}_{j,\xi} \ell_{j,\xi}^{(i)*} \|_F^2 + \lambda \cdot \sum_{j=1}^{q} \| \ell_{j}^{(i)} \|_{G}.
\end{equation*}
We apply this so that \eqref{eq:GIDL_opt} can be re-written as
\begin{equation} \label{eq:GIDL_opt_FT}
\begin{aligned}
  \underset{\ell_{j,\xi}^{(i)} ,\hat{\phi}_{j,\xi}}{\arg\min}\quad\quad &\sum_i \left( \sum_{\xi \in \hat{G}}  \frac{\mathrm{dim}(V_\xi)^2}{2} \| \hat{y}_{\xi}^{(i)} - \sum_{j=1}^{q} \hat{\phi}_{j,\xi} \ell^{(i)*}_{j,\xi} \|_F^2\right) + \lambda \cdot \sum_{j=1}^{q} \| \ell^{(i)}_{j} \|_{G}  \\
~~\text{s.t.} \quad\quad&  \sum_{\xi \in \hat{G}} \mathrm{dim}(V_\xi)^2 \| \hat{\phi}_{j,\xi} \|_F^2 = 1.
\end{aligned}
\end{equation}

\subsection{Band-limited Dictionary Learning}  \label{sec:bandlimited} 
In realistic scenarios, it is not practicable to use the entire sequence of irreducible representations $\hat{G}$ if the group $G$ is infinite.  Instead, a practical alternative is to restrict to a suitable finite subset of irreducibles $\hat{G}_{\mathrm{fin}}$.  We discuss the necessary modifications to our framework when working with $\hat{G}_{\mathrm{fin}}$.  In what follows, we let $L^2(\hat{G}_{\mathrm{fin}})$ denote the subspace of functions spanned by terms in $\hat{G}_{\mathrm{fin}}$ (specifically, replace $\hat{G}$ with $\hat{G}_{\mathrm{fin}}$ in \eqref{eq:noncomm_fourier}). In this context, it is highly pertinent to consider the analogue of band-limited functions in classical, real variable signal processing.

\textbf{Band-limited atomic norms.}  Recall that the atomic norm $\| \cdot \|_{G}$ is originally defined over the normed vector space $B(L^2(G))$.  When working with band-limited functions, the operators $\ell$ in our framework now naturally reside in $B(L^2(\hat{G}_{\mathrm{fin}}))$, which is finite dimensional.  We define $\| \cdot \|_{\hat{G}_{\mathrm{fin}}}$ to be the atomic norm induced by the block diagonal matrices
\begin{equation} \label{eq:bandlimited_atomicset}
\left\{ \pm 
\left( \begin{array}{ccc}
\ddots & & \\ & \rho_{\xi}(g)^* & \\ & & \ddots 
\end{array} \right)_{\xi \in \hat{G}_{\mathrm{fin}} }
 : g \in G \right\}. 
\end{equation}

On the surface, the atomic norm $\| \cdot \|_{\hat{G}_{\mathrm{fin}}}$ appears to coincide with the atomic norm $\| \cdot \|_{\hat{G}}$, applied only on the irreducibles in $\hat{G}_{\mathrm{fin}}$.  In fact, this intuition can be made formal.  More concretely:
\begin{proposition} \label{thm:bothdefinitionsequal}
Let $\hat{H} \subset \hat{G}$ be any subset of irreducibles (not necessarily finite).  Let $\| \cdot \|_{\hat{H}}$ be the atomic norm induced by \eqref{eq:bandlimited_atomicset}, with $\hat{H}$ in place of $\hat{G}_{\mathrm{fin}}$.  Then
\begin{equation*} 
\| t \|_{\hat{H}} = \inf \{ \| z \|_{\hat{G}} : P_{B(L^2(\hat{H}))}(z) = t , z \in B(L^2(\hat{G})) \}.
\end{equation*}
Here, $P$ denotes unitary projection onto $B(L^2(\hat{H}))$.
\end{proposition}

In particular, the utility of Proposition \ref{thm:bothdefinitionsequal} is to relate $\| \cdot \|_{\hat{G}_{\mathrm{fin}}}$, the atomic norm we deploy in practice, with $\| \cdot \|_{\hat{G}}$, which is the atomic norm our conceptual dictionary learning problem is based on.  We prove Proposition \ref{thm:bothdefinitionsequal} in Appendix \ref{sec:proof_alternatenorm_charac}.

\textbf{Modified dictionary learning procedure.}  The band-limited version of our proposed dictionary learning problem is as follows: in the formulation \eqref{eq:GIDL_opt_FT}, (i) we replace sums over $\xi \in \hat{G}$ with sums over $\xi \in \hat{G}_{\mathrm{fin}}$, and (ii) we replace $\| \cdot \|_G$ with $\| \cdot \|_{\hat{G}_{\mathrm{fin}}}$.  In particular, the $\hat{\phi}$'s are only supported on $\hat{G}_{\mathrm{fin}}$, while the $\ell$'s reside in $B(L^2(\hat{G}_{\mathrm{fin}}))$.

\textbf{Choosing $\hat{G}_{\mathrm{fin}}$.} The practitioner has some degree of freedom in selecting $\hat{G}_{\mathrm{fin}}$ -- in many settings, a sensible approach is to use some notion of `frequency' or `complexity' to rank the irreducibles.  For instance, one might list the irreducibles according to dimension, and only retain those with dimension up to some bandwidth.  In some cases, using dimension as the sole criteria may be insufficient -- for instance, all the irreducibles of SO(2) are one-dimensional (see Section \ref{sec:so2}).  In such settings, a more refined criteria is necessary; in the example of SO(2), ranking the irreducibles of SO(2) by its frequency would be a natural choice. On a more general Lie group, a suitable norm (such as the $H^1$ norm) of the character functions would be appropriate.

In what follows, our discussion applies to any bandwidth-limited $\hat{G}_{\mathrm{fin}}$.

\subsection{Algorithm} \label{sec:algo}

Our algorithm for solving \eqref{eq:GIDL_opt_FT} is to perform an alternating minimization scheme, in which we fix one variable and we update the other.  The algorithm is valid regardless of whether our enumeration is $\hat{G}$ or a truncation $\hat{G}_{\mathrm{fin}}$.  

\noindent\makebox[\linewidth]{\rule{\textwidth}{0.4pt}}
\textbf{Initialize.} Initialize an initial estimate of the functions $\{ \phi_j \}$, that are scaled to be unit-Frobenius norm. \\
\textbf{Step (i):} We fix the Fourier coefficients $\{ \hat{\phi}_{j,\xi} \}$ and we update the operators $\ell^{(i)}_{j,\xi}$.  This leads to the solution of a convex program, which is further separable in the index $i$ (denoting data), resulting in the following
\begin{equation} \label{eq:sparsecode}
 \underset{\{\ell_{j}\}_{j=1}^{q} }{\arg\min}\qquad  \sum_{\xi \in \hat{G}}  \frac{\mathrm{dim}(V_\xi)^2}{2} \| \hat{y}_{\xi} - \sum_{j=1}^{q} \hat{\phi}_{j,\xi} \ell_{j,\xi}^* \|_F^2 + \lambda \cdot \sum_{j=1}^{q} \| \ell_{j} \|_{G}  \qquad 1 \leq i \leq n.
\end{equation}
\textbf{Step (ii):} We fix the operators $\hat{\ell}^{(i)}_{j}$ and we update the Fourier coefficients $\{ \hat{\phi}_{j,\xi} \}$.  This leads to a least squares instance that is further separable in the irreducible representations $\xi$, leading to the following problem
\begin{equation} \label{eq:leastsquares}
\underset{\hat{\phi}_{j,\xi}}{\arg\min}~~\sum_i  \| \hat{y}_{\xi}^{(i)} - \sum_{j=1}^{q} \hat{\phi}_{j,\xi} \ell^{(i)*}_{j,\xi} \|_F^2, \qquad \xi \in \hat{G}.
\end{equation}
By taking the derivative with respect to $\hat{\phi}_{j,\xi}$ we have the first order condition
\begin{equation}
\sum_{i,j'} \ell^{(i) *}_{j,\xi} \ell^{(i)}_{j',\xi} \hat{\phi}^*_{j',\xi} = \sum_i \ell^{(i) *}_{j,\xi} \hat{y}_{\xi}^{(i)*}, \qquad 1 \leq j \leq q.
\end{equation}
Consequently, $\hat{\phi}$ is specified as follows -- here, $+$ denotes the pseudoinverse
\begin{equation} \label{eq:leastsquaresupdatespan}
\left( \begin{array}{c}
\hat{\phi}_{1,\xi}^* \\ \vdots \\ \hat{\phi}_{q,\xi}^*
\end{array} \right) = 
\left( \begin{array}{ccc}
\sum_{i} \ell^{(i)*}_{1,\xi} \ell^{(i)}_{1,\xi} & \ldots & \sum_{i} \ell^{(i)*}_{1,\xi} \ell^{(i)}_{q,\xi} \\
\vdots & & \vdots \\
\sum_{i} \ell^{(i)*}_{q,\xi} \ell^{(i)}_{1,\xi} & \ldots & \sum_{i} \ell^{(i)*}_{q,\xi} \ell^{(i)}_{q,\xi}
\end{array} \right)^{+}
\left( \begin{array}{c}
\sum_i \ell^{(i)*}_{1,\xi} \hat{y}_{\xi}^{(i)*} \\ \vdots \\ \sum_i \ell^{(i)*}_{q,\xi} \hat{y}_{\xi}^{(i)*}
\end{array} \right).
\end{equation}
\textbf{Step (iii):} Scale the updated functions $\{ \phi_j \}_{j=1}^{q}$ to have norm one
\begin{equation}
\hat{\phi}_{j,\xi} \leftarrow \frac{\hat{\phi}_{j,\xi}}{ \sqrt{ \sum_{\xi \in \hat{G}} \mathrm{dim}(V_\xi)^2 \|\hat{\phi}_{j,\xi}\|_F^2 }}.
\end{equation}
\textbf{Output.}  After desired number of iterations, we output the Fourier coefficients $\{ \hat{\phi}_{j,\xi} :1 \leq j \leq q, \xi \in \hat{G} \}$.  These coefficients specify a collection of functions $\{\hat{\phi}_j \}_{j=1}^{q} \subset L^2(\hat{G})$.  As a reminder, each of these $\hat{\phi}_j$ represent a canonical choice within an equivalence class.  In particular, the dictionary we learn is specified by $\{ \pm \tau(g) \hat{\phi}_j : g \in G, 1 \leq j \leq q \}$.   \\
\noindent\makebox[\linewidth]{\rule{\textwidth}{0.4pt}}

We make some quick observations regarding the computational complexity of the above procedure.  First, Step (i) solves for \eqref{eq:sparsecode} across each data-point, and hence the complexity is linear in the number of data-points $n$.  The complexity of {\em each instance} of \eqref{eq:sparsecode} on the problem dimensions is more delicate to characterize.  The ambient dimension of the $\ell$-s is $d_0 := \sum_{\xi \in \hat{G}} \mathrm{dim}(V_\xi)^2$ -- nominally, the dependency should be $O(\mathrm{poly}(d_0))$.  However, the squared loss term decouples across the $\xi$-s; in particular, the terms $\ell_{\xi}$ only interact via the function $\| \cdot \|_G$ applied to $\ell$.  This suggests the possibility of developing algorithms that split along such a factorization.  Second, Step (ii) solves for a least squares system, and has complexity $O(n+(q d_0)^3)$.  The $O(n)$ arises from computing the matrix and the vector sums the RHS of \eqref{eq:leastsquares}, and the $O(q d_0)^3$ arises from solving a linear system in the concatenated $\hat{\phi}$'s whose dimension is $q d_0$.

\subsection{Projecting to homogenous spaces}
The learned functions $\{ \hat{\phi}_j :1 \leq j \leq q \}$ reside in $L^2(G)$.  Our last step is to reverse our initial lifting procedure to obtain our desired dictionaries as functions in $L^2(\mathcal{X})$.

Concretely, for each $\hat{\phi}_j \in L^2(G)$, we define $\tilde{\phi}_j \in L^2(\mathcal{X})$ by
\begin{equation} \label{eq:proj_to_homogenous}
\tilde{\phi}(x) = \hat{\phi} (g) \qquad \text{ where } \qquad x_0 = g^{-1} x.
\end{equation}
Here, $x_0 \in \mathcal{X}$ is our choice of `origin' made earlier in Section \ref{sec:lifts}.

First, we claim that \eqref{eq:proj_to_homogenous} is well-defined; that is, $\hat{\phi} (g) = \hat{\phi} (h)$ whenever $g^{-1} x = h^{-1} x$.  To see why, we point out that a consequence of our construction of $f_y$ in \eqref{eq:homogenousspace_lift} is that it satisfies the property
\begin{equation} \label{eq:cosetstructure}
f_y (g) = f_y (gz) \qquad \text{ for all } \qquad z \in \mathrm{Stab}(x_0).
\end{equation}
It is straightforward to check that the set of functions satisfying \eqref{eq:cosetstructure} forms a vector subspace of $L^2(G)$.  The relationship \eqref{eq:cosetstructure} (via \eqref{eq:biaction}) translates to the following in terms of Fourier coefficients
\begin{equation} \label{eq:cosetstructure_fourier}
   \rho_\xi (z) \hat{f} (\xi) = \hat{f} (\xi) \quad \text{ for all } \quad z \in \mathrm{Stab}(x_0).
\end{equation}
Now, from \eqref{eq:leastsquaresupdatespan}, we observe that the $\hat{\phi}_{\xi}$'s are obtained by right multiplication on the matrices $\hat{y}_{\xi}$ (the conjugate transpose flips the order).  In particular, this means that the updated $\hat{\phi}$'s also satisfy \eqref{eq:cosetstructure_fourier}, which in turn implies well-definedness.

In practical implementations, the presence of numerical errors means that there will be variations in the value of $\hat{\phi} (h)$ across choices of input $h$ satisfying $x_0 = h^{-1} x$.  As such, it is good practice to define $\tilde{\phi}$ as an average across all choices of $h$ satisfying $x_0 = h^{-1} x$.  First, note that $x_0 = h^{-1} x$ if and only if $h^{-1}g \in \mathrm{Stab}(G)$.  In other words, the set of all $h$'s satisfying $x_0 = h^{-1} x$ is a coset of $\mathrm{Stab}(G)$.  Since $G$ is compact, $\mathrm{Stab}(x_0)$ is also compact, and hence one can define $\tilde{\phi}$ as follows
\begin{equation*}
\tilde{\phi}(x) = \int_{h :  x_0 = h^{-1} x } \hat{\phi} (h) ~ d\mu(h) = \int_{z \in \mathrm{Stab}(x_0)} \hat{\phi} (gz) ~ d\mu(z).
\end{equation*}

\subsection{Extensions to non-homogenous spaces}
We make some brief remarks concerning extensions of our framework to settings where $\mathcal{X}$ is {\em not} a homogenous $G$-space; that is, $G$ does not act transitively on $\mathcal{X}$.  In such settings, the appropriate extension of our framework is to view $\mathcal{X}$ as a union of orbits of $G$.  In particular, each data-point and each dictionary generating atom is a {\em collection of functions}, indexed by the $G$-orbits.  The group $G$ then acts on these functions jointly. 

We illustrate this with an example:  Suppose we have a collection of noisy images of a common object with a common center but unknown rotation.  The task is to learn a small set of representative images for this collection.  In this example, the appropriate choice of group to model this task is $SO(2)$; however, $SO(2)$ does {\em not} act on $\mathbb{R}^2$ transitively.  This is not an issue -- we simply view $\mathbb{R}^{2}$ as the union of circles centered at the origin, with $SO(2)$ acting on these orbits jointly.  

The ideas laid out in our framework extend to such settings in an analogous fashion.  We, however, do not discuss these extensions in this paper.

\section{Approximation Theoretic Guarantees} \label{sec:approx}

So far, we have described a framework for learning dictionaries for functions over groups that are invariant to the group action.  While our techniques are applicable to infinite groups, it is the finite dimensional truncations as outlined in Section \ref{sec:bandlimited} that are of practical interest.  In this light, a basic question is: What is the relationship between dictionaries obtained by solving the dictionary learning problem in infinite dimensional space with the dictionaries obtained by solving the same problem restricted to a subset of irreducibles?  This is the main topic of this section.  

To answer the above questions, we study a slightly modified version of the dictionary learning problem in \eqref{eq:GIDL_opt_FT}.  More precisely, consider the following:
\begin{equation} \label{eq:GIDL_opt_FT_lagrangian}
\underset{\ell_{j,\xi}^{(i)} ,\hat{\phi}_{j,\xi}}{\arg\min} \sum_i \left( \sum_{\xi \in \hat{H}}  \frac{\mathrm{dim}(V_\xi)^2}{2} \| \hat{y}_{\xi}^{(i)} - \sum_{j=1}^{q} \hat{\phi}_{j,\xi} \ell^{(i)*}_{j,\xi} \|_F^2 + \lambda \cdot \sum_{j=1}^{q} \| \ell^{(i)}_{j} \|_{\hat{H}} \right) + \sum_{j} \mu_j \cdot \left( \sum_{\xi \in \hat{H}} \mathrm{dim}(V_\xi)^2 \| \hat{\phi}_{j,\xi} \|_F^2 \right).
\end{equation}
Here, the parameters $\mu_j$ are {\em non-negative} penalty parameters and are to be viewed as being {\em fixed} in the above.  As before, $\phi_j \in L^2(\hat{H})$ are functions, $\ell_j^{(i)} : B(L^2(\hat{H}))$ are linear operators.  

The formulation \eqref{eq:GIDL_opt_FT_lagrangian} differs from \eqref{eq:GIDL_opt_FT} in that the constraint on the $\ell_2$-norm on $\phi_j$ now appears in the objective with $\mu_j$ as the corresponding \emph{Lagrange multiplier}.  This modification is made out of analytical considerations:  Generally speaking, there is a correspondence between the constrained formulation in \eqref{eq:GIDL_opt_FT} and the Lagrangian formulation in \eqref{eq:GIDL_opt_FT_lagrangian}.  By that, we mean that there is a choice of parameters $\{\mu_j\}_{j=1}^{q}$ (dependent on the data) such that the solution to \eqref{eq:GIDL_opt_FT_lagrangian} coincides to the solution to \eqref{eq:GIDL_opt_FT}.  

Given a subset of irreducibles $\hat{H} \subseteq \hat{G}$, we let $\mathrm{OPT-}\hat{H}$ denote the optimal value of \eqref{eq:GIDL_opt_FT_lagrangian}.  Our first result describes how the optimal value $\mathrm{OPT-}\hat{H}$ obtained using a partial set of irreducibles $\hat{H}$ in \eqref{eq:GIDL_opt_FT_lagrangian} compares with the optimal value $\mathrm{OPT-}\hat{G}$ obtained using the full set of irreducibles $\hat{G}$.

\begin{theorem} \label{thm:objval_increasing}  
Let $\{ \hat{H}_k \} \subseteq \hat{G}$ be an increasing subsequence (in the inclusion sense) such that $\hat{H}_k \uparrow \hat{G}$.  Then $\mathrm{OPT-}\hat{H}_k \uparrow \mathrm{OPT-}\hat{G}$.
\end{theorem}

Let $\hat{\Phi}^{\hat{H}} := \{ \hat{\phi}^{\hat{H}}_{1},\ldots,\hat{\phi}^{\hat{H}}_{q} \} \subset L^2(\hat{H})$ be the optimal dictionary to \eqref{eq:GIDL_opt_FT_lagrangian}.  Note that every function in $L^2(\hat{H})$ can be naturally embedded as function in $L^2(\hat{G})$; with a (minor) abuse of notation, we also let $\hat{\Phi}^{\hat{H}}$ denote such an embedding.  A basic question is: How does the solution $\hat{\Phi}^{\hat{H}}$ compare with $\hat{\Phi}^{\hat{G}}$ the optimal dictionary obtained using the full set of irreducibles, as functions in $L^2(\hat{G})$?

The first basis of comparison is in terms of objective value.  Concretely, we consider the objective value upon substituting $\hat{\Phi}^{\hat{H}}$ into \eqref{eq:GIDL_opt_FT_lagrangian} over the full set of irreducibles, reproduced in the following for clarity
\begin{equation} \label{eq:GIDL_opt_FT_lagrangian_full}
\underset{\ell_{j,\xi}^{(i)} ,\hat{\phi}_{j,\xi}}{\arg\min} \sum_i \left( \sum_{\xi \in \hat{G}}  \frac{\mathrm{dim}(V_\xi)^2}{2} \| \hat{y}_{\xi}^{(i)} - \sum_{j=1}^{q} \hat{\phi}_{j,\xi} \ell^{(i)*}_{j,\xi} \|_F^2 + \lambda \cdot \sum_{j=1}^{q} \| \ell^{(i)}_{j} \|_{\hat{G}} \right) + \sum_{j} \mu_j \cdot \left( \sum_{\xi \in \hat{G}} \mathrm{dim}(V_\xi)^2 \| \hat{\phi}_{j,\xi} \|_F^2 \right).
\end{equation}

\begin{theorem} \label{thm:objval_decreasing}  
Let $\{ \hat{H}_k \} \subseteq \hat{G}$ be an increasing subsequence (in the inclusion sense) such that $\hat{H}_k \uparrow \hat{G}$.  Then $\mathrm{OBJ-}\hat{H}_k \downarrow \mathrm{OPT-}\hat{G}$.
\end{theorem}

Theorems \ref{thm:objval_increasing} and \ref{thm:objval_decreasing} concerns the \emph{objective value}.  In many cases, one may also interested in understanding the convergence behavior of the dictionary atoms.  To make such statements, we require a few more ingredients.  

First, we need to quantify the difference between pairs of dictionaries in a fashion that accounts for group action.  Given two dictionaries $\Phi := \{ {\phi}_1, \ldots, {\phi}_q \}$ and $\tilde{\Phi} := \{ \tilde{\phi}_1, \ldots, \tilde{\phi}_q \}$, we define the distance as the maximum difference between pairs of elements, modulo equivalences over permutations and the group transformations
\begin{equation} \label{eq:dict_distance}
D(\Phi,\tilde{\Phi}) : = \min_{\pi, g_j \in G} \max_j \| \phi_j - g_j \tilde{\phi}_{\pi(j)} \|_2.
\end{equation}
Here, $\pi : \{1,\ldots,q\} \mapsto \{1,\ldots,q\}$ is a permutation.  

Second, we need an additional assumption guaranteeing some form of uniqueness of the optimal solution.  

\begin{definition}[Canonical Uniqueness]
Let $\hat{\Phi}$ be a dictionary that attains the infimum OPT in \eqref{eq:GIDL_opt_FT_lagrangian}.  We say that $\hat{\Phi}$ is \emph{canonically unique} if any $\tilde{\Phi}$ that attains an objective value of OPT$+\epsilon$ guarantees $\tilde{\Phi}$ is close to $\hat{\Phi}$ in that
\begin{equation*}
D(\hat{\Phi} , \tilde{\Phi} ) \leq \delta(\epsilon) ,
\end{equation*}
for some function satisfying $\delta(\epsilon) \downarrow 0$ as $\epsilon \downarrow 0$.
\end{definition}

\begin{proposition} \label{thm:dictapprox}
Let $\hat{\Phi}^{\hat{G}}$ be an optimal dictionary to \eqref{eq:GIDL_opt_FT_lagrangian}, and suppose it is canonically unique.  Suppose $\{ \hat{H}_k \} \subseteq \hat{G}$ is an increasing subsequence (in the inclusion sense) such that $\hat{H}_k \uparrow \hat{G}$. 
Then
\begin{equation*}
D(\hat{\Phi}^{\hat{H}_k}, \hat{\Phi}^{\hat{G}}) \rightarrow 0  \qquad \text{ as } \qquad k \rightarrow \infty.
\end{equation*}
\end{proposition}

One might wonder if the assumption on canonical uniqueness is unnecessarily strong or vacuous.  To address the first point, we briefly remark that if there are multiple dictionaries that achieves the same objective value in \eqref{eq:GIDL_opt_FT_lagrangian} but are not equivalent, the solution to the truncated problem could in principle also be close to any of these dictionaries.  It would then be not possible to state any form of convergence.  To address the second point, in Section \ref{sec:so3} we demonstrate numerical experiment using synthetic data for the group SO(3).  In the experiment, our algorithm succeeds at recovering the ground truth dictionary (up to an equivalence), thus showing that recovering the optimal dictionary is in principle possible, at least for synthetically generated data.

\section{Learning SO(2) Invariant Dictionaries}

In following sections, we describe the full program of learning a dictionary that is invariant to some groups.  For each example, we characterize the irreducible representations and we provide spectrahedral descriptions of the associated atomic norm.  

\subsection{Spectrahedral descriptions of the SO(2) orbitope}

A \emph{spectrahedron} is a convex set that can be described as the intersection of the cone of positive semidefinite (PSD) matrices with an affine sub-space
\begin{equation*}
\{ X : X \succeq 0, \langle A_i , X \rangle = \mathrm{tr}(A_i X) = b_i, 1 \leq i \leq d \} \subseteq \mathbb{S}^{d}.
\end{equation*}
A \emph{semidefinite program} (SDP) is a convex programming instance in which we minimize a linear functional over spectrahedra 
\begin{equation*}
\min \quad \langle C, X \rangle \quad \mathrm{s.t} \quad \langle A_i, X \rangle = b_i,  1 \leq i \leq d,\quad X \succeq 0.
\end{equation*}
It extends \emph{linear programs} (LPs), which are convex programming instances in which we minimize a linear functional over polyhedra.  SDPs are an important class of convex programs because they possess expressive modeling power and we have tractable algorithms for solving them \cite{NesNem:94,Ren:01}.

\textbf{Caratheodory Orbitopes.}  Let $v^{(j)}(\theta)$ be the following vector
\begin{equation*}
v^{(j)}(\theta) = \left(\exp( 0 \cdot i \theta), \exp( 1 \cdot i \theta) ,\ldots, \exp(j \cdot i \theta) \right)^T \in \mathbb{R} \times \mathbb{C}^j \subset \mathbb{C}^{(j+1)}.
\end{equation*}

\noindent \emph{Notation.}  In what follows, we will denote the first coordinate by $0$.  The entry will typically be real.





\begin{definition}[Caratheodory Orbitope, \cite{SSS:11}]

The Caratheodory Orbitope is defined as
\begin{equation*}
\mathcal{C}_j := \mathrm{conv}( \{ v^{(j)}(\theta) : \theta \in [0, 2\pi ) \} ).
\end{equation*}
\end{definition}

The following result gives a description of $\mathcal{C}_j$ as the feasible region of a SDP:

\begin{proposition} \label{thm:caratheodory_psdtoep}
We have
\begin{equation*}
\mathcal{C}_j = \{ z : z = Z_{:0}, Z_{0,0} = 1, Z  \text{ is PSD Hermitian Toeplitz} \} \subset \mathbb{C}^{j+1}.
\end{equation*}
Here, $Z_{:0}$ denotes the leftmost column of $Z$.
\end{proposition}

The proof of this result relies on the following result whereby every Toelitz PSD matrix admits a Vandermonde decomposition:
\begin{proposition}[\cite{Caratheodory:11,CarFej:11,Toeplitz:11}] \label{thm:vandermonde}
Let $X$ be PSD (Hermitian) Toeplitz.  Then $X$ admits a Vandermonde decomposition of the form $X = VDV^\dagger$, where $V$ is a Vandermonde matrix and $D$ is a diagonal matrix with positive entries.
\end{proposition}

For a complete proof of Proposition \ref{thm:caratheodory_psdtoep}, see, for instance, \cite{SSS:11}.


\subsection{Learning SO(2)-invariant Dictionaries} \label{sec:so2}  

Consider learning a dictionary for signals over an interval that is invariant under continuous cyclic shifts, or learning a dictionary for images that is invariant under rotations about the origin.  The appropriate group for modeling such symmetries is SO(2).  In what follows, we describe the full program for learning SO(2)-invariant dictionaries.  In particular, the regular Fourier series play an important role.

\textbf{Irreducible representations.}  Concretely, consider the cyclic group $G = S^1 := \{ \exp( i \theta) : \theta \in [0,2\pi) \}$ where multiplication is addition (in the argument) modulo $2\pi$.  The irreducibles representations are the homomorphisms
\begin{equation*}
\rho_n : \exp(i \theta) \mapsto \exp(i n \theta), \qquad n \in \{ \ldots, -2,-1,0,1,2,\ldots\}.
\end{equation*}
These are all one-dimensional.  Recall that every function $f$ over $S^1$ can be expressed as follows
\begin{equation*}
f(\theta) = \sum_{n = - \infty}^{\infty} c_n \exp(i n \theta), \qquad c_n \in \mathbb{C}.
\end{equation*}
This is precisely the regular Fourier series for periodic functions.  Here, $c_n$ are the Fourier coefficents. 

\textbf{Atomic norm.}  Consider an element $g := \exp(i \alpha) \in S^{1}$.  Then $g$ acts on $f(\theta)$ as follows
\begin{equation*}
g \cdot f(\theta) = f( g^{-1} \theta) = \sum_{n = - \infty}^{\infty} c_n \exp(i  n (\theta - \alpha )) = \sum_{n = - \infty}^{\infty} c_n \exp( -i \alpha n) \exp(i  n \theta) .
\end{equation*}
In terms of the Fourier coefficients, the action of $g$ is as follows
\begin{equation*}
g : c_n \mapsto c_n \exp( -i \alpha n).
\end{equation*}
Consequently, the atomic set is the collection of diagonal operators of the following form
\begin{equation} \label{eq:SO2_atomicset}
\left \{ \mathrm{diag}\left( ( \ldots, \exp( -1 \cdot i \alpha), \exp( 0 \cdot i \alpha), \exp( 1 \cdot i \alpha), \ldots ) \right) : \theta \in [0,2\pi) \right \}.
\end{equation}
Here, we say that an operator $t \in L^2(S^1) \rightarrow L^2(S^1)$ is diagonal if it acts on the Fourier coefficients as follows
\begin{equation*}
t: ( \ldots  , c_{-1}, c_0, c_1, \ldots ) \mapsto ( \ldots, t_{-1,-1} c_{-1}, t_{0,0} c_0, t_{1,1} c_1, \ldots ).
\end{equation*}
The linear span of \eqref{eq:SO2_atomicset} are diagonal operators such that the $j$-th and $-j$-th entry are conjugate.  

Consider a truncated basis $\tilde{G} = \{ \exp(-iN \theta), \ldots, \exp(iN \theta) \}$.  The truncated atomic norm $\| \cdot \|_{\tilde{G}}$ is the convex function induced by the set
\begin{equation*}
\mathrm{conv}( \{ ( \exp(-N \cdot i \theta) , \ldots, \exp(-1 \cdot i \theta),   \exp(0 \cdot i \theta), \exp(1 \cdot i  \theta), \ldots, \exp(N \cdot i \theta) )^T : \theta \in [0,2\pi ) \} ).
\end{equation*}
Given a vector in $\mathbb{C}^{2j+1}$ such that the $j$-th and the $-j$-th coordinate are conjugate pairs (we let $0$ denote the middle coordinate), we denote $\mathbf{x}^\downarrow : = (x_{0},x_{1},\ldots,x_{j})^T \in \mathbb{C}^{j+1}$.  By combining \eqref{thm:caratheodory_psdtoep} with a simple reflection argument, we have the following (see also \cite{Soh:20})


\begin{equation*}
\left\|  \mathbf{x} \right\|_{\tilde{G}} ~~=~~ \inf \{  \lambda_+ + \lambda_- :  \mathbf{x}^\downarrow = \lambda_+ \mathbf{z}^{+} - \lambda_- \mathbf{z}^{-} , ~ \mathbf{z}^{+} \in \mathcal{C}_j, ~ \mathbf{z}^{-} \in \mathcal{C}_j, \lambda_+, \lambda_- \geq 0 \}.
\end{equation*}

The atomic norm is expressible via a SDP.  To do so, we instead consider the cone generated by $\mathcal{C}_j$, which is obtained by omitting the normalization constraint $Z_{0,0}=1$ in Proposition \ref{thm:caratheodory_psdtoep}.  This leads to the following \cite{Soh:20}
\begin{equation*}
\begin{aligned}
\|  \mathbf{x} \|_{\tilde{G}} ~~=~~ \inf z_{0,+} + z_{0,-} ~~\mathrm{s.t.}~~  &  \left( \begin{array}{c} x_0 \\ x_1 \\ \vdots \\ x_j \end{array} \right) = 
\left( \begin{array}{c} z_{0,+} \\ z_{1,+} \\ \vdots \\ z_{j,+} \end{array} \right) -
\left( \begin{array}{c} z_{0,-} \\ z_{1,-} \\ \vdots \\ z_{j,-} \end{array} \right) \\
& ~~
\left( \begin{array}{ccc}
z_{0,+} & \overline{z_{1,+}} & \ldots \\
z_{1,+} & * & * \\
\vdots & * & *
\end{array} \right),
\left( \begin{array}{ccc}
z_{0,-} & \overline{z_{1,-}} & \ldots \\
z_{1,-} & * & * \\
\vdots & * & *
\end{array} \right) \quad \text{ is PSD Toeplitz}
\end{aligned}
\end{equation*}

\textbf{Connections to Convolutional Dictionary Learning.}  Our preceding discussion recovers precisely the same framework of learning a \emph{continuously} shift invariant dictionary, which was earlier described in \cite{Soh:20}.  Interestingly, while the work in \cite{Soh:20} motivates the procedure as a continuous analog of convolutional dictionary learning -- these are dictionaries invariant to integer shifts or the cyclic group $\mathbb{Z} / n \mathbb{Z}$, we derive the same procedure by considering functions over SO(2).  This is also somewhat expected -- $\mathbb{Z} / n \mathbb{Z}$ are the finite sub-groups of the cyclic group over $S^1$, so learning SO(2) invariant dictionary is in a precise sense the continuous limit of convolutional dictionary learning (c.f. \cite{SFB:19}).

\section{Learning O(2) Invariant Dictionaries} \label{sec:o2} 

In this section, we describe the full program of learning a dictionary that is invariant to O(2).  One example of an application is for images that are invariant under rotations about the origin and \emph{reflections} about any line through the origin.  Our discussion serves as a companion to the earlier discussion for the group SO(2) in Section \ref{sec:so2}.

\textbf{Irreducible representations.}  Concretely, consider the matrix group with elements of the form
\begin{equation*}
\left\{ 
R(\theta) := \left( \begin{array}{cc} \cos \theta & \sin \theta \\ - \sin \theta & \cos \theta \end{array} \right), 
L(\theta) := \left( \begin{array}{cc} \cos \theta & \sin \theta \\ \sin \theta & - \cos \theta \end{array} \right) : \theta \in [0,2\pi) 
\right\}.
\end{equation*}
The irreducible representations are the mappings $\rho_k : (R(\theta),L(\theta)) \mapsto ( \rho_{R,k}( \theta),\rho_{L,k}( \theta) )$, where
\begin{equation*}
\rho_{R,k}( \theta) = \left( \begin{array}{cc}
\exp( i k \theta)  &  \\ & \exp( -i k \theta)
\end{array} \right), \qquad
\rho_{L,k}( \theta) = \left( \begin{array}{cc}
& \exp( i k \theta) \\ \exp( -i k \theta) &
\end{array} \right).
\end{equation*}
The irreducible representations are all two-dimensional, except the trivial one corresponding to $k=0$.  The Fourier coefficients are $2\times 2$ complex matrices, and every function over O(2) can be expressed as
\begin{equation} \label{eq:o2_fourier}
f(g) = c_0 +  \sum_{k=1}^{\infty} 2 \mathrm{tr}( C_k \rho_{k}(g) ), \qquad g \in O(2),
\end{equation}
where $C_k$ are the matrix-valued Fourier coefficients.

\textbf{Atomic norm.}  Truncate the basis by taking the upper limit in \eqref{eq:o2_fourier} to be $N$.  Then the action by an element in O(2) is the block operator
\begin{equation*}
(1, \rho_{1}(\theta), \ldots, \rho_{N}(\theta)).
\end{equation*}
Consider an operator of the form
\begin{equation*}
Z := (z_0, Z_{1}, \ldots , Z_{N} ) \in \mathbb{R} \times (\mathbb{C}^{2 \times 2})^{\otimes N}.
\end{equation*}
The atomic norm is given by
\begin{equation*}
\begin{aligned}
\left\| Z \right\|_{\tilde{G}} ~~= & ~~ \inf \quad z_{L,+,0} + z_{L,-,0} + z_{R,+,0} + z_{R,-,0} \\
& ~~ \mathrm{s.t.} \quad \mathbf{z}_{L} = \mathbf{z}_{L,+} - \mathbf{z}_{L,-}, ~~ \mathbf{z}_{R} = \mathbf{z}_{R,+} - \mathbf{z}_{R,-} \\
&  \quad \quad \quad Z = \left( z_{R,0} + z_{L,0} ,
\left( \begin{array}{cc}
z_{R,1} & z_{L,1} \\ z_{L,-1} & z_{R,-1}
\end{array} \right), \ldots, \left( \begin{array}{cc}
z_{R,N} & z_{L,N} \\ z_{L,-N} & z_{R,-N}
\end{array} \right) \right) \\
& \quad \quad \quad  \mathbf{z}_{L,+}, \mathbf{z}_{L,-}, \mathbf{z}_{R,+}, \mathbf{z}_{R,-} \in \mathcal{C}_{N}
\end{aligned}.
\end{equation*}

\section{Learning SO(3) Invariant Dictionaries} \label{sec:so3}

We describe the program when the group is SO(3).  The ideas we lay out provide the foundation for learning dictionaries for spherical images that are rotation and translation invariant~\cite{CarlosSO(3):18,Kon07,Kondor:08}.

\textbf{Euler angles.}  The group SO(3) can be parameterized by three angles $\alpha$, $\beta$, and $\gamma$ known as the \emph{Euler angles}.  Concretely, every element $R \in \mathrm{SO}(3)$ decomposed as a rotation about the $z$-axis, followed by a rotation about the $y$-axis, and then a rotation about the $z$-axis
\begin{equation*}
R(\alpha,\beta,\gamma) = Z(\alpha) Y(\beta) Z(\gamma), \qquad \text{ where } \qquad \alpha,\gamma \in [0,2\pi), \beta \in [0,\pi).
\end{equation*}
Here, $Z$ and $Y$ are rotation matrices about the $z$ and $y$ axes given as follows
\begin{equation*}
Z(\theta) = 
\left(\begin{array}{ccc}
\cos(\theta) & \sin (\theta) & \\ -\sin(\theta) & \cos(\theta) & \\ & & 1 
\end{array}\right), \qquad 
Y(\beta) = 
\left(\begin{array}{ccc}
\cos(\theta) & & \sin (\theta) \\ & 1 & \\ -\sin(\theta) & & \cos(\theta)
\end{array}\right).
\end{equation*}

\textbf{Irreducible representations.}  The irreducible representations of SO(3) are known as Wigner $D$-matrices, and these are indexed by $j \in \{0,1,\ldots\}$.  For a fixed $j$, the corresponding irreducible representation is a square matrix of dimension $(2j+1) \times (2j+1)$ with the following entries
\begin{equation} \label{eq:wigner_terms}
D^{(j)}_{m,m^\prime} (\alpha, \beta,\gamma) = \exp( - i m \alpha) \times d^{(j)}_{m^\prime, m} (\beta) \times \exp( - i m^\prime \gamma).
\end{equation}
Here, $m$ and $m^\prime$ are the indices of the matrix and they span $\{-j, -j+1,\ldots j \}$.  The scalar $d^{(j)}_{m^\prime, m}$ is defined as
\begin{equation} \label{eq:wigner_terms2}
d^{(j)}_{m^\prime, m} (\beta) = \sum_{k} c_k(j,m,m^\prime) \times (\cos \frac{\beta}{2})^{2j-2k+m-m^\prime} (\sin \frac{\beta}{2})^{2k-m+m^\prime},
\end{equation}
where
\begin{equation} \label{eq:wigner_terms3}
c_k = (-1)^{k-m+m^\prime} \frac{\sqrt{(j+m)!(j-m)!(j+m^\prime)!(j-m^\prime)!}}{(j+m-k)!(j-k-m^\prime)!(k-m+m^\prime)!k!}.
\end{equation}
In the expression \eqref{eq:wigner_terms2}, the sum over the index $k$ is such that the arguments in the factorial in \eqref{eq:wigner_terms3} are non-negative.

\subsection{Multivariate Extensions of the Caratheodory Orbitope}

The irreducible representations of SO(3) contains trigonometric functions of \emph{three} angles.  To describe the atomic norm, we need a multivariate analog of the Caratheodory orbitope.


\begin{definition}[Hermitian Block Toeplitz]  Let $T_{x_1,\ldots, x_r, y_1,\ldots, y_r} \in (\mathbb{C}^{n_1 \times n_2 \times \ldots \times n_r})^{\otimes 2}$ be a tensor.  We say that $T$ has a \emph{block Toeplitz} structure if it satisfies the following properties
\begin{enumerate}
\item $T_{x_1,\ldots, x_r, y_1,\ldots, y_r} = T_{\tilde{x}_1,\ldots, \tilde{x}_r, \tilde{y}_1,\ldots, \tilde{y}_r}$ if $x_k - y_k = \tilde{x}_k - \tilde{y}_k$ for some $k$, and $(x_l, y_l) = (\tilde{x}_l , \tilde{y}_l)$ for all $l \neq k$,
\item $T_{x_1,\ldots, x_r, y_1,\ldots, y_r} = \overline{T_{\tilde{x}_1,\ldots, \tilde{x}_r, \tilde{y}_1,\ldots, \tilde{y}_r}}$ if $x_k - y_k = -(\tilde{x}_k - \tilde{y}_k)$ for some $k$, and $(x_l, y_l) = (\tilde{x}_l , \tilde{y}_l)$ for all $l \neq k$.
\end{enumerate}
Equivalently, $T$ is Hermitian Block Toeplitz if the $n_k \times n_k$-dimensional matrix obtained by fixing all except the $k$-th and the $r+k$-th coordinate is Hermitian Toeplitz as a matrix for all choices of $k$, $1\leq k \leq r$, and all choices of coordinates $x_1,\ldots,\widehat{x_k},\ldots,x_r, y_1,\ldots,\widehat{y_k},\ldots,y_r$.  (Here, $\widehat{x_k}$ means we omit $x_k$.)
\end{definition}

Given a tensor $T \in (\mathbb{C}^{n_1 \times n_2 \times \ldots \times n_r})^{\otimes 2}$ and a tensor $u \in \mathbb{C}^{n_1 \times \ldots \times n_r}$, we define the tensor $T(u)$ obtained via a  contraction over the indices $y_1,\ldots,y_r$  as follows
\begin{equation*}
[T(u)]_{x_1,\ldots,x_r} = \sum_{y_1,\ldots,y_r} T_{x_1,\ldots,x_r,y_1,\ldots,y_r} u_{y_1,\ldots,y_r}.
\end{equation*}
Equivalently, if we re-arrange (vectorize) the entries of $u$ so that it is a $n_1 \times \ldots \times n_r$-dimensional vector and the entries of $T$ to be a $n_1 \times \ldots \times n_r$ by $n_1 \times \ldots \times n_r$ square matrix, then the regular matrix-vector product of $T$ and $u$ is precisely the re-arranged (vectorized) analog of $T(u)$.

Given two tensors $u,v \in \mathbb{C}^{n_1 \times \ldots \times n_r}$, we define the Hermitian inner product
\begin{equation*}
\langle u, v \rangle = \sum_{x_1,\ldots,x_r} u_{x_1,\ldots,x_r} \overline{v_{x_1,\ldots,x_r}}.  
\end{equation*}

\begin{definition}[Positive Semidefinite Tensor]  With a slight abuse of notation, we say that a tensor $T \in (\mathbb{C}^{n_1 \times n_2 \times \ldots \times n_r})^{\otimes 2}$ is \emph{positive semidefinite} (PSD) if for any tensor $u \in \mathbb{C}^{n_1 \times \ldots \times n_r}$, one has $\langle u, V (u) \rangle \geq 0$.  Equivalently, $T$ is PSD if the $(n_1 \ldots n_r) \times (n_1 \ldots n_r)$-dimensional matrix obtained by rearranging the entries of $T$ so that the first $r$ coordinates form a single coordinate and the next $r$ coordinates form another coordinate is PSD as a matrix.
\end{definition}


\begin{definition}[Multivariate Caratheodory Orbitope]
\begin{equation} \label{eq:convexhull_outerproducts}
\mathcal{C}_{j_1,\ldots,j_k} := \mathrm{conv} \left( \big\{ v^{(j_1)}(\theta_1) \otimes \ldots \otimes v^{(j_k)}(\theta_k)  : 0 \leq \theta_1,\ldots,\theta_k \leq 2 \pi \big\} \right).
\end{equation}    
\end{definition}

Our goal is to try derive a description of $\mathcal{C}_{j_1,\ldots,j_k}$, or at least an outer approximation, as the feasible region of a SDP.  as such, we mimic the ideas in Propositions \ref{thm:caratheodory_psdtoep}
and \ref{thm:vandermonde}.

\begin{proposition}\label{thm:blocktoep}
The set $\mathcal{C}_{j_1,\ldots,j_k}$ is contained in
\begin{equation} \label{eq:convhull_psdtoep}
\begin{aligned}
C_{\mathrm{Block Toep.}} ~:=~ \big\{ z \in \mathbb{C}^{(j_1+1)\times \ldots (j_r+1)} ~ : & ~ z = Z_{*,\ldots,*, 0,\ldots, 0} , \\
& ~ Z \in (\mathbb{C}^{(j_1+1)\times \ldots (j_r+1)})^{\otimes 2} \text{ is PSD block Toeplitz}, \\
& ~  Z_{x_1,\ldots,x_r, x_1,\ldots, x_r} = 1 \text{ for all } x_1,\ldots,x_r  \big\} .
\end{aligned}
\end{equation}
\end{proposition}

\begin{proof}[Proof of Proposition \ref{thm:blocktoep}]   First we check that points of the form $z = v^{(j_1)}(\theta_1) \otimes \ldots \otimes v^{(j_k)}(\theta_k)$ are inside \eqref{eq:convhull_psdtoep}.  Let $Z = v^{(j_1)}(\theta_1) \otimes \ldots \otimes v^{(j_k)}(\theta_k) \otimes \overline{v^{(j_1)}(\theta_1)} \otimes \ldots \otimes \overline{v^{(j_k)}(\theta_k)}$.  By construction, we have $ Z_{x_1,\ldots,x_r,y_1 = x_1,\ldots,y_r = x_r} = 1 $, and $Z$ is PSD block Toeplitz.  To see that $z = Z_{*,\ldots,0, \ldots}$, we note that the $0$-th coordinate of $\overline{v^{(j_1)}(\theta_1)}, \ldots, \overline{v^{(j_k)}(\theta_k)}$ are all ones.  Combining the previous two statements, we have $z$ is contained in \eqref{eq:convhull_psdtoep}.  Last, since \eqref{eq:convhull_psdtoep} is convex, the convex hull of \eqref{eq:convexhull_outerproducts} is also contained in \eqref{eq:convhull_psdtoep}.
\end{proof}

In contrast with Proposition \ref{thm:caratheodory_psdtoep}, the description in \eqref{eq:convhull_psdtoep} is only guaranteed to be an outer approximation of $\mathcal{C}_{j_1,\ldots,j_k}$.  We are not aware of a tractable description of $\mathcal{C}_{j_1,\ldots,j_k}$, and we have reason to believe the set may in fact be intractable to express.  First, recall that the convex hull of the collection of unit-norm rank-one $(2k)$-tensors is the \emph{tensor nuclear norm} ball (see Proposition 3.1, \cite{FL:18}).  Second, if we stick with the ideas of Proposition \ref{thm:blocktoep} by lifting to the outer product of \eqref{eq:convexhull_outerproducts} with its conjugate, we obtain a collection of unit-norm rank-one tensors with block Toeplitz structure and whose diagonals equal to one.  It is natural to then consider the intersection of the tensor nuclear norm ball with the same affine constraints; however, it is NP-hard to compute the $(2k)$-tensor nuclear norm over $\mathbb{C}$ for $k \geq 2$ (see Theorem 8.10 of \cite{FL:18}).  


\textbf{Tightness of SDP relaxation of multivariate Caratheodory orbitope.}  But just how good is the relaxation in Proposition \ref{thm:blocktoep}?  We investigate this question using numerics.  In the following, we generate random tensors of the form
\begin{equation} \label{eq:randomtensormodel}
T := \sum_{i=1}^{r} c_i = (1/r) \exp( i \theta)  v^{(n)}(\theta_{i,1}) \otimes v^{(n)}(\theta_{i,2}) \otimes v^{(n)}(\theta_{i,3}) \quad \text{where} \quad \theta, \theta_{i,k} \sim \mathsf{Unif}[0, 2\pi).
\end{equation}

In what follows, we compute the Minkowski functional with respect to the relaxation obtained via Proposition \ref{thm:blocktoep} -- this can be done, for instance, via the following SDP
\begin{equation} \label{eq:sdp_relaxation_blocktoep}
\min ~~ (1/2)t + (1/2) Z_{0,0,0,0,0,0} \quad \mathrm{s.t.} \quad 
\left( \begin{array}{cc}
t & \mathrm{vec}(T)^T \\ \mathrm{vec}(T) & \mathrm{mat}(Z)
\end{array} \right) \succeq 0, Z \text{ is block Toeplitz}.
\end{equation}
In general, because \eqref{eq:sdp_relaxation_blocktoep} computes the Minkowski function with respect to a convex set that is an outer convex relaxation of $\mathcal{C}_{j_1,\ldots,j_k}$, the optimal solution to \eqref{eq:sdp_relaxation_blocktoep} bounded above by the atomic norm evaluated at $T$.

Next, $T$ is a linear sum of rank-one tensor.  There are $r$ terms, and each have a coefficient of $1/r$.  As such, the resulting atomic norm evaluated at $T$ is at most the sum of these coefficients, which is one.  In general, it may be smaller, especially for large $r$.

Based on these series of inequalities, if the optimal value to \eqref{eq:sdp_relaxation_blocktoep} is equal to one, we have an instance where the relaxation is tight.  In the following, we perform the following experiment:  For each pair $(n,r)$, $n\in \{1,2\}, r \in \{1,2,\ldots,10\}$, we generate $100$ random tensors and we count the number of instances the optimal value to \eqref{eq:sdp_relaxation_blocktoep} is $> 1 - 10^{-4}$ -- see Table \ref{table:so3relaxation}.  We record the percentage of instances where the optimal value equals $1$. 

\begin{table}[h]
\centering
\begin{tabular}{|c|c|c|c|c|c|c|c|c|c|c|}
\hline
$n = 1$ & $r = 1$ & $r = 2$ & $r = 3$ & $r = 4$ & $r = 5$ & $r= 6$ & $r=7$ & $r=8$ & $r=9$ & $r=10$ \\
\hline
\# successes & 100 & 96 & 81 & 56 & 36 & 18 & 4 & 2 & 0 & 0 \\
\hline \hline
$n = 2$ & $r = 1$ & $r = 2$ & $r = 3$ & $r = 4$ & $r = 5$ & $r= 6$ & $r=7$ & $r=8$ & $r=9$ & $r=10$ \\
\hline
\# successes & 100 & 100 & 98 & 95 & 93 & 82 & 81 & 81 & 77 & 69 \\
\hline
\end{tabular}
\caption{Quality of the convex relaxation of the SO(3) atomic norm.  The table indicates the fraction of instances where the optimal solution equals $1$, which (approximately) indicates how often the relaxation is tight.}
\label{table:so3relaxation}
\end{table}

We observe that, for small values of $r$, the frequency for which the optimal value is equal to one is rather high, which suggests that the convex relaxation in \eqref{eq:convhull_psdtoep} may in fact be quite tight.  As one increases $r$, the frequency drops.  We think that this is likely attributed to the fact that the atomic norm of $T$ is no longer equal to one, rather than indicative of the quality of exactness.

\subsection{SDP relaxation of atomic norm} \label{sec:sdp_atomicnorm_multivariate}

Recall that the atomic set of the irreducible representations is of the form
\begin{equation} \label{eq:so3_regularreps}
\left \{ 
\left( \begin{array}{ccc}
D^{(0)} (\alpha,\beta,\gamma) & &\\
& D^{(1)} (\alpha,\beta,\gamma) & \\
& & \ddots \\
\end{array} \right) :\alpha,\beta,\gamma \leq 2 \pi  \right \}.
\end{equation}
Subsequently, the atomic norm $\|\cdot\|_G$ is the operator norm induced by the convex hull of \eqref{eq:so3_regularreps}.  

Next, we describe the finite dimensional approximation of the atomic norm $\|\cdot\|_G$.  A natural way to truncate the Fourier basis according to the size of the Fourier coefficients.  Let $\tilde{G} : = \hat{G}_N$ be the basis spanned by the Wigner $D$-matrices for all indices $j \leq N$.  Subsequently, the associated atomic norm is induced by the following set
\begin{equation} \label{eq:so3_truncatednorm}
\mathrm{conv} \left( \left \{ 
\left( \begin{array}{ccc}
D^{(0)} (\alpha,\beta,\gamma) & &\\
& \ddots & \\
& & D^{(N)} (\alpha,\beta,\gamma) \\
\end{array} \right) :\alpha,\beta,\gamma \leq 2 \pi  \right \} \right).
\end{equation}

First, we define $V_{\alpha \beta \gamma}^{(j)} := v^{(j)}(\alpha) \otimes v^{(j)}(\beta) \otimes v^{(j)}(\gamma) \otimes \overline{v^{(j)}(\alpha)} \otimes \overline{v^{(j)}(\beta) } \otimes \overline{v^{(j)}(\gamma) }$.  Moving forward, we treat $\alpha$, $\beta$, and $\gamma$ as \emph{symbols} rather than specific numerical values.  Notice at this juncture that all the terms in the Wigner $D$-matrix for all indices $j \leq N$ appear as terms in the tensor $V_{\alpha \beta \gamma}^{(N)}$.  Therefore, there is a linear map $L^{(j)} : (\mathbb{C}^{N+1})^{\otimes 6} \rightarrow (\mathbb{C}^{2j+1})^{\otimes 2}$ such that
\begin{equation} \label{eq:defn_of_L}
D^{(j)} = L^{(j)} (V_{\alpha \beta \gamma}^{(N)}) .
\end{equation}
Let
\begin{equation*}
\tilde{\mathcal{C}}^{N}:=\mathrm{conv} \left(  \big\{ v^{(N)}(\alpha) \otimes v^{(N)}(\beta) \otimes v^{(N)}(\gamma)  : 0 \leq \alpha,\beta,\gamma \leq 2 \pi \big\} \right).    
\end{equation*}
By linearity, the convex hull of \eqref{eq:so3_truncatednorm} is $ \{ ( ~ L^{(0)} (x) , \ldots, L^{(N)} (x) ~)^T \,:\, x \in \tilde{\mathcal{C}}^{N} \}$.  Unfortunately, we are not aware of a tractable description of this set.  Instead, following Proposition \eqref{thm:blocktoep}, an outer approximation of this set is
\begin{equation} \label{eq:so3_relaxation}
\{ ( ~ L^{(0)} (x) , \ldots ,
L^{(N)} (x) 
 ~)^T \, : \, x \in C_{\mathrm{Toep}}^{(N)} \},
\end{equation}
where
\begin{equation*}
\begin{aligned}
C_{\mathrm{Toep}}^{(N)} = \big\{ Z^{+} - Z^{-} : & ~  Z^{+}_{x_1,\ldots,x_r,x_1 ,\ldots,x_r} = 1 , Z^{-}_{x_1,\ldots,x_r,x_1,\ldots,x_r} = 1 \text{ for all } x_1,\ldots, x_r , \\
& ~ Z^{+},Z^{-}  \text{ are PSD block Toeplitz} \big\}.
\end{aligned}
\end{equation*}

\subsection{Tightness of SDP relaxation of SO(3) orbitope} \label{sec:so3_supp} 
Equation \eqref{eq:so3_relaxation} provides an outer approximation of the SO(3) orbitope.  A natural question is, how does this approximation perform when deployed to learn a SO(3) invariant dictionary? 
In what follows, we evaluate the efficacy of the relaxation \eqref{eq:so3_relaxation} via a numerical experiment on synthetic data.  Specifically, for a fixed choice of $j \in \{0,1,2,\ldots\}$, our data is synthetically generated via the model
\begin{equation*}
y^{(i)} = D^{(j)} (\alpha^{(i)}, \beta^{(i)}, \gamma^{(i)}) \phi^\star.
\end{equation*}
In other words, the dictionary is generated by a single element $\{\phi^\star\}$.  Here, $\phi^\star \in \mathbb{C}^{(2j+1) \times (2j+1)}$ is a random complex matrix drawn from the normal distribution and later normalized to be of unit Frobenius norm, while $\alpha,\beta,\gamma \sim \mathsf{Unif}[0,2\pi)$.

In the first experiment, we generate $50$ data points for the choice $j=1$.  Note that because the data are $3\times 3$-dimensional \emph{complex} matrices, the ambient (real) dimension is $18$.  We apply the algorithm described Section \ref{sec:algo}.  To do so, we initialize with a random matrix $\phi_0$, normalized to be of unit norm.  We select our regularization parameter $\lambda = 0.1$.  (Remark: Our results are reflective of a wide range of choices of $\lambda$.)  To track progress of our algorithm, we measure the distance of the iterate from the true generator $\phi^\star$ using the distance measure \eqref{eq:dict_distance} specialized to this instance
$\mathrm{dist}(\phi,\phi^\star) = \inf_{\alpha,\beta,\gamma,\theta} \| \phi - \exp(i \theta) D(\alpha,\beta,\gamma) \phi^\star \|_F^2$.  Here, we include the unitary element $\exp(i \theta)$ because of sign.  Computing the distance measure is difficult in general -- in our implementation, we combine a grid search over $\alpha$, $\beta$, and $\gamma$ with a binary search over $\theta$.  

We show the progress of our algorithm in the left sub-plot of Figure \ref{fig:dictionaryerrors} (solid line with circle nodes).  First, we observe that the initial error estimate (after one iteration) is $\approx 0.8$.  This is useful because it suggests that two generic random dictionaries are in the ball-park distance of $0.8$ apart.  The final estimate has error $\approx 10^{-4}$.  

We compare our results with a baseline using vanilla dictionary learning -- this is akin to applying our framework but disregarding $G$ completely.  In the first instance, we learn using a dictionary element with a single element, and we show the distance of the iterate from $\phi^\star$ in the same sub-plot (solid line with bar nodes).  We notice that the final dictionary estimate attains an error of $\approx 0.1$, which substantially improves upon the initial estimate, but poorer than the solution obtained using our method.  In the second instance, we increase to using $5$ dictionary elements to see if the performance improves (dashed lines).  We use error bars to indicate the distance between the ground truth generator with all dictionary elements.  While the nearest dictionary elements are closer to $\phi^\star$ compared to the previous instance where we only permitted dictionary element, the remaining elements remain far away.  This is notwithstanding the fact that the dictionary is specified by $5$ distinct dictionary elements whereas our learned dictionary is specified by \emph{one} canonical element.

In the second experiment, we repeat the same experimental set-up but for $j=2$ (the number of data-points is $50$, and the choice of regularization parameter is $0.1$).  This time, the ambient (real) dimension is $50$, which is substantial compared to the amount of data.  Nevertheless, by incorporating the appropriate symmetries, our framework succeeds at recovering the correct $\phi^\star$.

\begin{figure}
\centering
\includegraphics[width=0.4\textwidth]{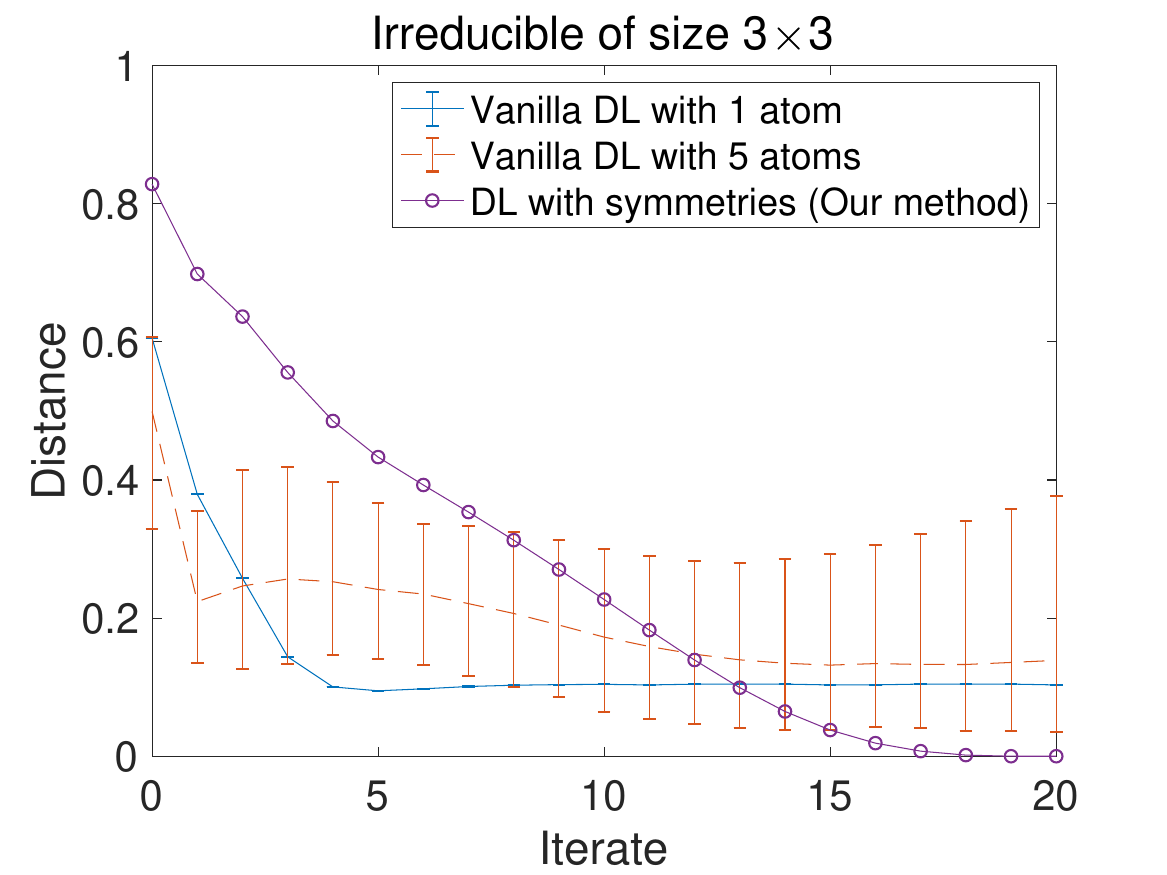}
\includegraphics[width=0.4\textwidth]{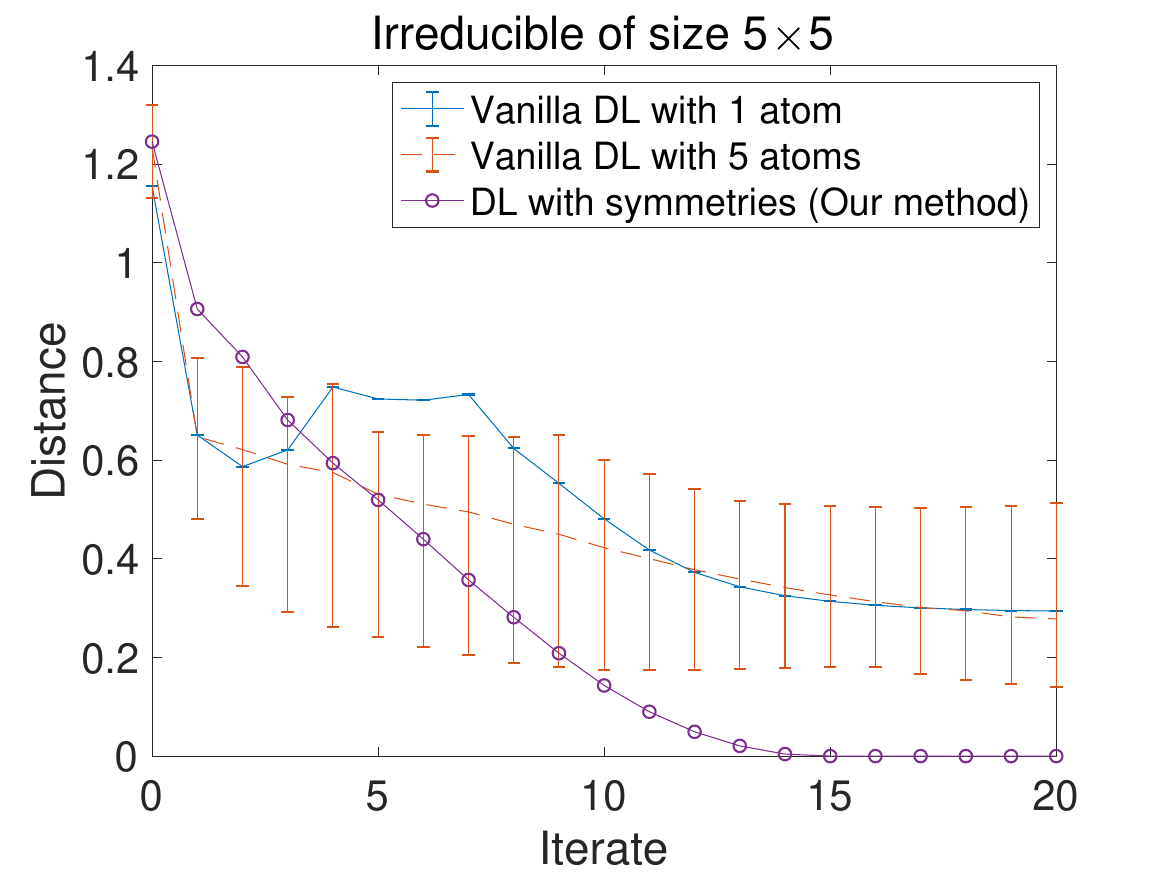}
\caption{Comparison of our framework for recovering a SO(3) dictionary with dictionary learning.  For dictionary learning, the line plot indicate the average distance from the true dictionary, while the bars indicate the minimum and the maximum distance.}
\label{fig:dictionaryerrors}
\end{figure}

In the third experiment, we repeat the same experimental set-up for $j \in \{1,2\}$, but with {\em noisy} data.  To each entry of the data matrix, we add gaussian complex noise $\epsilon_R + i \epsilon_C$, where $\epsilon_R,\epsilon_C$ are drawn from the normal distribution $\mathcal{N}(0,\sigma^2)$, and with varying levels of $\sigma$.  The purpose of this experiment is to investigate if our method is robust to noise.  

We present the results from this numerical experiment in Figure \ref{fig:noisyerrors}.  We make some brief remarks:  First, the recovery of the generating atom becomes substantially more difficult, but nevertheless possible especially at low levels of noise.  Generally speaking, we need more iterations to obtain a reasonable estimate of the generating atom.  The distance between the generating atom and the recovered estimate increases with the level of noise.  We emphasize that the number of data-points ($50$) is rather modest in comparison to the number of parameters to be estimated.  It would be interesting to see if the accuracy of the recovered atom improves with substantially more data.

\begin{figure}
\centering
\includegraphics[width=0.4\textwidth]{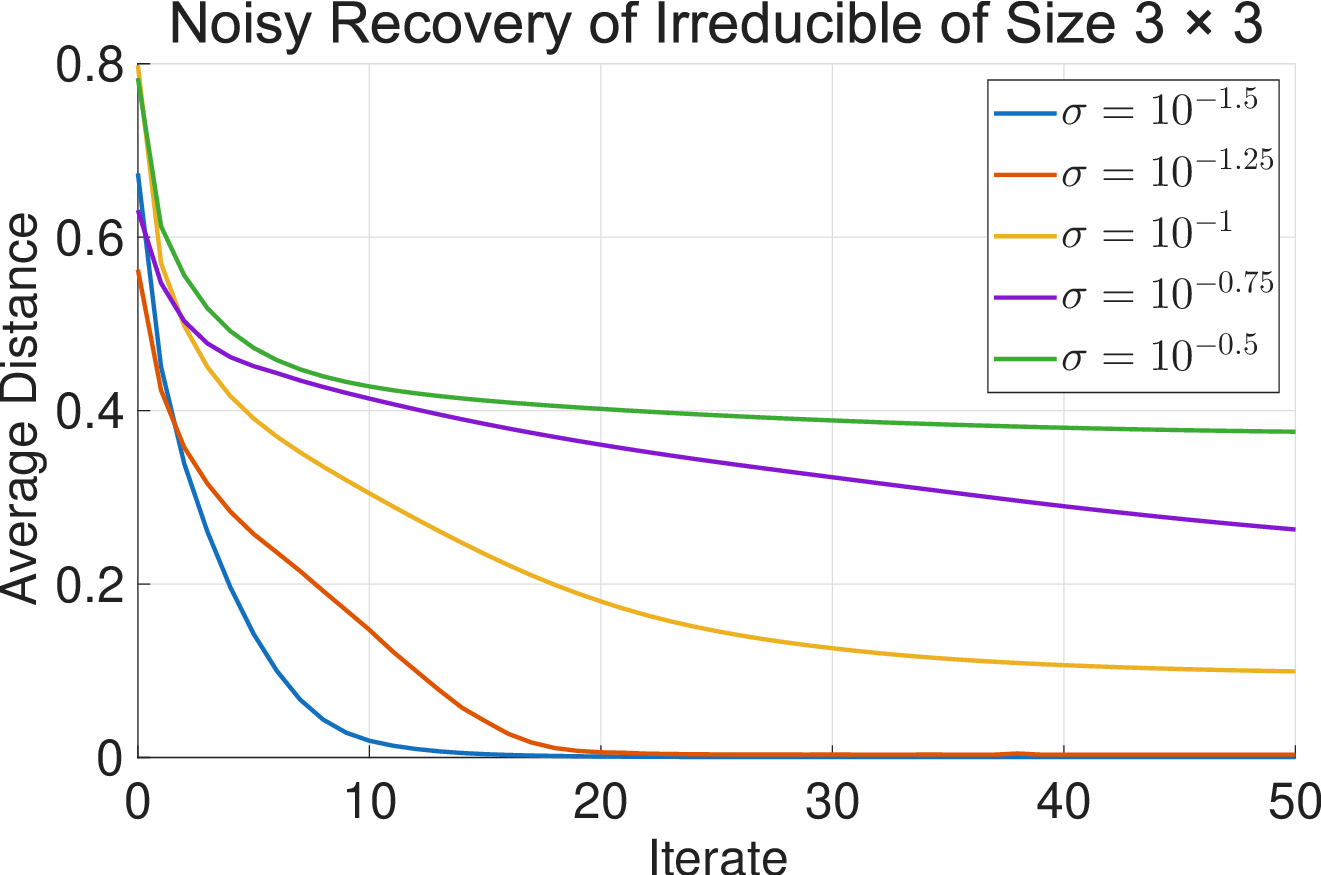}
\includegraphics[width=0.4\textwidth]{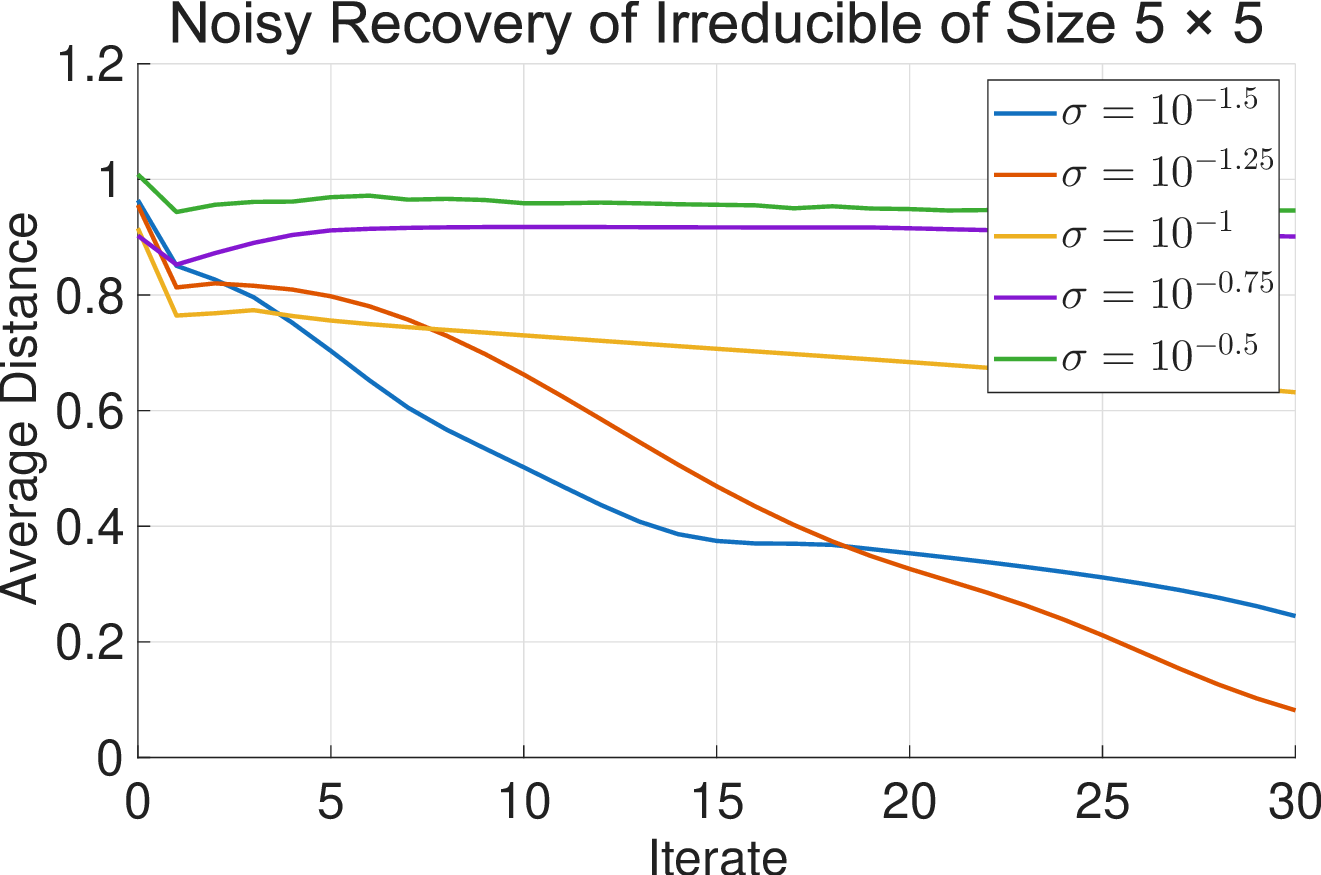}
\caption{Noisy recovery of irreducibles.  The experimental set-up is identical to that in Figure 
\ref{fig:dictionaryerrors}, but with noisy data.}
\label{fig:noisyerrors}
\end{figure}

In summary, the numerical experiments show that we were able to recover the ground truth dictionary elements correctly despite implementing a relaxation of the atomic norm in our procedure.  The numerical results also suggest that our method exhibits some level of robustness to noise.  These observations suggest an interesting direction in our framework -- namely, that applying a convex outer or inner approximation, and especially one that is computationally easier to express, may yield comparable results. 
\subsection{An alternating minimization scheme for learning a SO(3) invariant dictionary} \label{sec:so3cd}

In this subsection, we describe an alternative scheme based on alternating minimization for learning a SO(3) invariant dictionary, which we apply to a dataset comprising handwritten digits.  Our motivation for developing an alternative scheme comes from computational considerations:  The techniques developed in this paper specify a method for learning SO(3) invariant dictionaries by solving a SDP (as a sub-routine) whose number of variables grow on the order $O(N^3)$, where $N$ is the data dimension (note: in general, one would need about $O(\sqrt{N})$ Fourier coefficients to attain a data dimension of size $N$).  Unfortunately, the degree $3$ dependency means that solving the SDP becomes exorbitant rather quickly, even for more modestly-sized experiments.  The goal of this subsection is to demonstrate a practical algorithm that is capable of learning SO(3) invariant dictionaries for data of moderately large size.

\textbf{Alternating Minimization.}  The costliest aspect of our proposed dictionary learning procedure is the sparse coding step as in \eqref{eq:sparsecode}.  To circumvent the computational difficulties, we propose solving the following instance instead
\begin{equation} \label{eq:so3_coding_onesparse}
\underset{c,\alpha,\beta,\gamma}{\arg \min} ~~ \sum_{\xi \in \hat{G}} \| \hat{y}_\xi - \sum_{j=1}^{q} \hat{\phi}_{j,\xi} (c_j D^{(\xi)}_j (\alpha,\beta,\gamma))^* \|_F^2.
\end{equation}
Here, $\hat{y}_\xi$ and $ \hat{\phi}_{j,\xi}$ are Fourier coefficients corresponding to the data and the dictionary respectively, while $D_j \in U(L^2(G))$ is an irreducible representation.  In essence, \eqref{eq:so3_coding_onesparse} restricts the solution $\ell_j$ in \eqref{eq:sparsecode} to be the linear combination of a \emph{single} irreducible.  In terms of regular dictionary learning, we seek a code that is exactly one-sparse.  The main benefit of the formulation in \eqref{eq:so3_coding_onesparse} is that the number of variables is reduced substantially: we only have four variables for every dictionary element.  The price we pay, however, is that the optimization instance \eqref{eq:so3_coding_onesparse} is non-convex -- in general, we do not know how to solve \eqref{eq:so3_coding_onesparse} analytically.

Our strategy for obtaining high quality solutions to \eqref{eq:so3_coding_onesparse} is a coordinate descent-based approach.  Specifically, we minimize over the variables $\alpha,\beta,\gamma$, and $c$ separately while keeping all other variables fixed.  We do so in a cyclic fashion and in that order.  To minimize over the angles $\alpha$, $\beta$, $\gamma$, we perform a grid search, while the minimization over $c$ admits a closed form solution given by $c = \mathrm{tr}(Y(D \phi)^\dagger) / \| D \phi \|_F^2 $.  We performed  tests on our heuristic for solving \eqref{eq:so3_coding_onesparse}, and we found that the squared error loss seems to stagnate after about $5$ iterations.  In our numerical implementations, we apply $5$ iterations across these variables.  

\textbf{Numerical Experiment.}  As a proof of concept, we apply our dictionary learning procedure on handwritten digits from the MNIST dataset \cite{mnist} -- further details are provided in the Supplementary Materials.  We impose SO(3) invariance and hence the learned images, shown in Figure \ref{fig:dictionaryelements}, are invariant up to translations and rotations.  We observe resemblances between the digits $0$, $1$ and $5$, with the learned dictionary elements, while the resemblances between other digits and the corresponding learned elements are weaker.  We believe an explanation for our observations is that the Euclidean loss is suboptimal for image processing tasks, and that our results can be improved using better suited losses such as those based on optimal transport.

\begin{figure}
\centering
\includegraphics[width=0.09\textwidth]{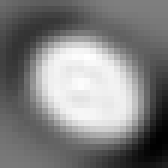}
\includegraphics[width=0.09\textwidth]{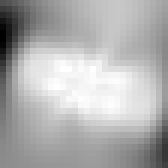}
\includegraphics[width=0.09\textwidth]{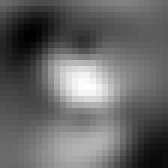}
\includegraphics[width=0.09\textwidth]{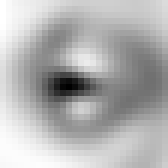}
\includegraphics[width=0.09\textwidth]{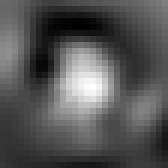}
\includegraphics[width=0.09\textwidth]{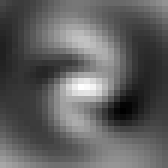}
\includegraphics[width=0.09\textwidth]{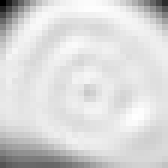}
\includegraphics[width=0.09\textwidth]{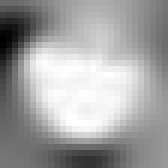}
\includegraphics[width=0.09\textwidth]{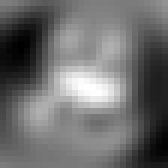}
\includegraphics[width=0.09\textwidth]{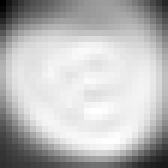}
\caption{Learned dictionary corresponding to digits $0$ to $9$ from left to right.}
\label{fig:dictionaryelements}
\end{figure}

\section{Concluding Remarks} \label{sec:conc}  In this paper, we study the problem of learning a dictionary invariant to a pre-specified group of symmetries through the lens of representation theory.  Using  the representation theory of compact groups, we describe a recipe for learning dictionaries invariant to compact group symmetries.  We briefly examine the harmonic analytical aspects of the dictionary learning problem, and we describe in detail the program for the groups SO(2), O(2), and SO(3).  

In what follows, we briefly remark on important future directions that stem from our work.

\textbf{Incorporating symmetries versus data augmentation.}  Symmetries and invariances in data arise very naturally in numerous application domains.  Many successful representation techniques take advantage of this basic observation in some form or another.  Fundamentally, this paper is guided by the philosophy that it is useful to incorporate the appropriate symmetries and invariances in the \emph{structure} of the data representation we learn.  Admittedly, the more typical response in similar settings is \emph{data augmentation} -- in simple terms, the data analyst appends transformed copies of the original dataset in the hope that the learned model has the appropriate invariant properties (for instance, one might hope a classifier predict the same outcome for any rotation of the same image).

The basic reason why data augmentation is widely used is that it is simple and cheap.  However, we have no control over whether the learned model is truly invariant -- it never truly `realizes' that it is intended to be invariant.  The computational methods we introduce here are powerful in that it learns representations that enforce invariance -- the price we pay is computation cost.  Are there application domains where the importance of incorporating the appropriate invariant structures in the learned data representation models sufficiently justify the additional computational cost required to train it?  This is a basic question in representation learning that warrants further study.

\textbf{Learning localized patches.}  A very closely variant of the problem formulation in \ref{eq:gidl_problem} is to view the dictionary elements as \emph{localized} patches or \emph{filters}.  In this formulation, we view the data as a \emph{single} vector $y \in \mathbb{R}^{\tilde{d}}$ residing in very high dimensions -- for instance, this could be a photograph or a single time series.  The goal is to learn a collection of filters $\{ \phi_j \}_{i=1}^{q} \subset \mathbb{R}^{\tilde{d}}$ that are only supported over $d \ll \tilde{d}$ consecutive coordinates -- this is why these patches are localized.

From a mathematical and a conceptual perspective, these problems are identical -- the techniques developed in this paper follow in a straightforward fashion to the more general setting.  From a computational perspective, the fact that these patches are embedded in a much higher dimensional space $\mathbb{R}^{\tilde{d}}$ introduces challenges that we need to address.  First, data dimension is effectively $\mathbb{R}^d$, and so $\rho(g)$ resides in a much higher dimensional space.  Second, our framework advocates working in the Fourier domain for computational reasons.  However, a spatially localized signal will not be localized in the frequency domain.  It would be interesting to bring in techniques from numerical optimization and computational harmonic analysis to address these challenges.

\textbf{SO(3) orbitopes.}  An outstanding unresolved question in this paper is if the SO(3) orbitope is tractable to express.  Closely related to this set is the tensor nuclear norm ball restricted to tensors with block Toeplitz structure.  While the tensor nuclear norm is known to be intractable to describe \cite{FL:18}, there are principled methods for deriving tractable convex relaxations -- one such scheme is based on the sum-of-squares hierarchy in which these relaxations are specified as the solution of a SDP \cite{Nie:17}.  It would be interesting to further investigate these connections.

\textbf{Conic Programming Descriptions of Orbitopes.}  In Section \ref{sec:intro}, we stated that one of our main contributions is to provide a recipe for describing $\rho(g)$ explicitly.  This still leaves an unresolved question -- namely -- how do we provide descriptions of the convex hull of $\{ \rho(g) : G \}$ that are amenable to optimization?  Answering this question in full generality is most certainly out of reach.  A more modest goal could be to provide a procedure for obtaining conic programming-based descriptions of structured classes of orbitopes, for instance, when the generators are known.

\section{Acknowledgments}

The authors wish to thank the two excellent anonymous reviewers whose comments have substantially improved the manuscript.

\appendix

\section{Proof of Proposition \ref{thm:bothdefinitionsequal}} \label{sec:proof_alternatenorm_charac}

In this section we prove Proposition \ref{thm:bothdefinitionsequal}.  To improve readability, we introduce the following:
\begin{equation} \label{eq:alternative_atomicnorm}
\| t \|^+_{\hat{H}} = \inf \{ \| z \|_{\hat{G}} : P_{B(L^2(\hat{H}))}(z) = t , z \in B(L^2(\hat{G})) \}.
\end{equation}

We restate Proposition \ref{thm:bothdefinitionsequal} equivalently as follows:
\begin{proposition} 
    Let $\hat{H} \subset \hat{G}$  be a subset of irreducibles.  Let $t \in B(L^2(\hat{H}))$.  Then
    $$
    \| t \|_{\hat{H}} = \| t \|_{\hat{H}}^+.
    $$
\end{proposition}

\begin{proof}  There are two parts to this proof.

[$\| t \|_{\hat{H}} \leq \| t \|_{\hat{H}}^+$]:  First, suppose $\| t \|_{\hat{H}}^+=1$.  In particular, this means that there exists $z \in B(L^2(\hat{G}))$ such that $P(z) = t$, and $\|z\|_{\hat{G}} = 1$.  (In principle, the infimum need not be attained.  If this happens, we refine these steps to a sequence $z_i$ such that $\|z_i\|_{\hat{G}} \rightarrow 1$.)   Here, $P$ denotes the projection onto $L^2(\hat{H})$.  By definition of being in the convex hull, this means that
\begin{equation} \label{eq:proof31_equation1}
z = \sum_{i} \theta_i ( \rho_{\xi} (g_i) )_{\xi \in \hat{G}}.    
\end{equation}
Here, $\theta_i$ are weights satisfying $\sum |\theta_i| = 1$ and they arise from $t$ belonging to the convex hull.  Note that we use $( \rho_{\xi} (g_i) )_{\xi \in \hat{G}}$ to denote the block diagonal operator in $B(L^2(\hat{G}))$ whose $\xi$-th block is $\rho_{\xi} (g_i)$.  Do also note that because the atomic norm is defined with respect to $\pm \rho(g)$, the weights $\theta_i$ are permitted to be signed.  Last, it is possible that $z$ belongs to the closure, but not the convex hull.  Should that happen, we refine this proof by replacing $z$ with a sequence of elements in the convex hull converging to $z$.

To \eqref{eq:proof31_equation1}, we apply the projection operator $P$ on both sides.  In particular, 
$$
P [ (\rho_{\xi} (g))_{\xi' \in \hat{G}} ]_{\xi} = 
\begin{cases}
\rho_{\xi} (g) \text{ if } \xi \in \hat{H} \\
0 \text{ otherwise }
\end{cases}.
$$
This can be summarized as
$$
P(z) = \sum_{i} \theta_i P[ ( \rho_{\xi} (g_i) )_{\xi \in \hat{G}}] = \sum_{i} \theta_i ( \rho_{\xi} (g_i) )_{\xi \in \hat{H}} = t.
$$
That is to say, $t \in \mathrm{conv} ( \{ \pm (\rho_{\xi})_{\xi \in \hat{H}} : g \in G \})$, and hence $\| t \|_{\hat{H}} \leq 1$.

[$\| t \|_{\hat{H}}^+ \leq \| t \|_{\hat{H}}$]:  Suppose $\| t \|_{\hat{H}} = 1$.  This means that $t \in \mathrm{conv} ( \{ \pm (\rho_{\xi})_{\xi \in \hat{H}} : g \in G \})$.  Let's write
$$
t = \sum_{i} \theta_i (\rho_{\xi}(g_i))_{\xi \in \hat{H}}
$$
where $\sum |\theta_i| = 1$.  Now we define
$$
z = \sum_{i} \theta_i (\rho_{\xi}(g_i))_{\xi \in \hat{G}} \in B(L^2(\hat{G})).
$$
In other words, $z$ is defined by extending all the terms $(\rho_{\xi}(g_i))_{\xi \in \hat{H}}$ to the full set of irreducibles $\hat{G}$.  Now, it is clear that $P(z) = t$.  Moreover, in the above we have $z \in \mathrm{conv} ( \{ \pm (\rho_{\xi})_{\xi \in \hat{G}} : g \in G \})$.  Therefore $\| t \|_{\hat{H}}^+ \leq \| z \|_{\hat{G}} \leq 1$.
\end{proof}

\section{Proofs of Theorems \ref{thm:objval_increasing} and \ref{thm:dictapprox}}

In this section we provide the proofs of Theorems \ref{thm:objval_increasing} and \ref{thm:dictapprox}.  In what follows, $\hat{H} \subseteq \hat{G}$ is a subset of irreducibles.

We begin with some observations.  First, given $f \in L^2(\hat{G})$, we denote its projection onto $L^2(\hat{H})$ by $P_{L^2(\hat{H})} (f)$.  Recall from \eqref{eq:noncomm_fourier} that the Fourier series expansion of $f$ is given by 
\begin{equation*}
f(x) ~=~ \sum_{\xi \in \hat{G}} \mathrm{dim}(V_\xi) ~ \left \langle \hat{f}(\xi), \rho_\xi(x)^* \right\rangle_\xi ~=~ \sum_{\xi \in \hat{G}} \mathrm{dim}(V_\xi) ~ \mathrm{tr}_{V_\xi}( \hat{f}(\xi) \rho_\xi(x) ) .
\end{equation*}
Because the Fourier functions form an orthogonal basis of $L^2(\hat{G})$, it follows that the Fourier series expansion of $P_{L^2(\hat{H})} (f)$ is simply the truncated expansion up to terms in $\hat{H}$
\begin{equation} \label{eq:objval_inc_proj1}
[P_{L^2(\hat{H})} (f)](x) ~=~ \sum_{\xi \in \hat{H}} \mathrm{dim}(V_\xi) ~ \left \langle \hat{f}(\xi), \rho_\xi(x)^* \right\rangle_\xi ~=~ \sum_{\xi \in \hat{H}} \mathrm{dim}(V_\xi) ~ \mathrm{tr}_{V_\xi}( \hat{f}(\xi) \rho_\xi(x) ) .
\end{equation}

Second, and in a similar fashion, given an operator $t \in B(L^2(\hat{G}))$, we denote its projection onto $B(L^2(\hat{H}))$ by $P_{B(L^2(\hat{H}))}(t)$.  Suppose $t$ is a block diagonal operator that acts on the Fourier coefficients by right multiplication
\begin{equation*} 
t : \hat{f}(\xi) \mapsto \hat{f}(\xi) t_\xi, \qquad \text{ where } \quad  t_\xi \in \mathbb{C}^{\mathrm{dim}(V_\xi) \times \mathrm{dim}(V_\xi)}, \quad \xi \in \hat{G}.
\end{equation*}
In particular (and as a reminder), a consequence of Peter-Weyl's is that the learned operators $\ell$ will obey such structure.  Then the projection of $t$ will also obey the same block diagonal structure, but restricted to the terms in $\hat{H}$
\begin{equation} \label{eq:objval_inc_proj2}
P_{B(L^2(\hat{H}))}(t) : \hat{f}(\xi) \mapsto \hat{f}(\xi) t_\xi,  \qquad \xi \in \hat{H}.
\end{equation}

In the next result as well as what follows, we will use the definition \eqref{eq:alternative_atomicnorm}.  As we shall see, the definition of $\|\cdot \|_{\hat{H}}$ via \eqref{eq:alternative_atomicnorm} is more convenient for analysis.

\begin{proposition} \label{thm:atomicnorm_increasing}
Let $t \in B(L^2(\hat{G}))$.  Then 
\begin{equation*}
\| P_{B(L^2(\hat{H}))}(t) \|^+_{\hat{H}} \leq \| t \|^+_{\hat{G}}.
\end{equation*}
\end{proposition}

\begin{proof}[Proof of Proposition \ref{thm:atomicnorm_increasing}]
By noting that $t$ is a feasible solution to the argument in the definition in \eqref{eq:alternative_atomicnorm}, we have
\begin{equation*}
 \| P_{B(L^2(\hat{H}))}(t) \|^+_{\hat{H}} = \inf \{ \| z \|_{\hat{G}} : P_{B(L^2(\hat{H}))}(z) = t , z \in B(L^2(\hat{G})) \} \leq  \| t \|_{\hat{G}} .
\end{equation*}
\end{proof}

\begin{proposition} \label{thm:dictapprox_obj}  Let $\hat{H} \subseteq \hat{G}$ be subset of irreducibles.  Let $\mathrm{OPT-}\hat{H}$ be the optimal value of \eqref{eq:GIDL_opt_FT_lagrangian}, and let and $\mathrm{OPT-}\hat{G}$ be the optimal value of \eqref{eq:GIDL_opt_FT_lagrangian_full} (as a reminder, this is the same as \eqref{eq:GIDL_opt_FT_lagrangian} but over the full set of irreducibles $\hat{G}$).  Then $\mathrm{OPT-}\hat{H} \leq \mathrm{OPT-}\hat{G}$.
\end{proposition}

\begin{proof}[Proof of Proposition \ref{thm:dictapprox_obj}]
Let $\hat{\phi}^{\hat{G}} \in L^2(\hat{G})$ and $\ell^{\hat{G}} \in B(L^2(\hat{G}))$ be a pair of feasible solutions to \eqref{eq:GIDL_opt_FT_lagrangian_full}.  Let the corresponding objective value be $\mathrm{VAL}$.  

Next, we substitute the pair of solutions $P_{L^2(\hat{H})}(\hat{\phi}^{\hat{G}})$ and $P_{B(L^2(\hat{H}))}(\ell^{\hat{G}})$ to \eqref{eq:GIDL_opt_FT_lagrangian_full}, and we compare its corresponding objective value with $\mathrm{VAL}$.  

First, by comparing the first term in \eqref{eq:GIDL_opt_FT_lagrangian} and in \eqref{eq:GIDL_opt_FT_lagrangian_full}, we have
\begin{equation*}
\begin{aligned}
& \sum_{\xi \in \hat{H}} \mathrm{dim}(V_\xi)^2 \| \hat{y}_{\xi} - \sum_{j=1}^{q} [P_{L^2(\hat{H})}(\hat{\phi}^{\hat{G}})]_{j,\xi} [P_{B(L^2(\hat{H}))}(\ell^{\hat{G}})]^*_{j,\xi} \|_F^2 \\
=~ & \sum_{\xi \in \hat{H}} \mathrm{dim}(V_\xi)^2 \| \hat{y}_{\xi} - \sum_{j=1}^{q} \hat{\phi}^{\hat{G}}_{j,\xi} \ell^{\hat{G}*}_{j,\xi} \|_F^2 \\ 
\leq ~ & \sum_{\xi \in \hat{G}} \mathrm{dim}(V_\xi)^2 \| \hat{y}_{\xi} - \sum_{j=1}^{q} \hat{\phi}^{\hat{G}}_{j,\xi} \ell^{\hat{G}*}_{j,\xi} \|_F^2.
\end{aligned}
\end{equation*}
The first equality is a consequence of \eqref{eq:objval_inc_proj1} and \eqref{eq:objval_inc_proj2}.  

Second, by Proposition \ref{thm:atomicnorm_increasing}, we have $\| P_{B(L^2(\hat{H}))}(\ell^{\hat{G}}) \|^+_{\hat{H}} \leq \| \ell^{\hat{G}} \|^+_{\hat{G}}$.  

Third, we have 
\begin{equation*}
\sum_{\xi \in \hat{H}} \mathrm{dim}(V_\xi)^2 \| P_{L^2(\hat{H})}(\hat{\phi}^{\hat{G}})_{\xi} \|_F^2 ~=~ \sum_{\xi \in \hat{H}} \mathrm{dim}(V_\xi)^2 \| \hat{\phi}^{\hat{G}}_{\xi} \|_F^2  ~\leq~ \sum_{\xi \in \hat{G}} \mathrm{dim}(V_\xi)^2 \| \hat{\phi}^{\hat{G}}_{\xi} \|_F^2.
\end{equation*}

By summing these three inequalities, the pair of solutions $P_{L^2(\hat{H})}(\hat{\phi}^{\hat{G}})$ and $P_{B(L^2(\hat{H}))}(\ell^{\hat{G}})$ attain an objective value at most $\mathrm{VAL}$ in \eqref{eq:GIDL_opt_FT_lagrangian}.  The result follows by taking the infimum over all feasible $\hat{\phi}^{\hat{G}}$ and $\ell^{\hat{G}}$.
\end{proof}

Let $\hat{\Phi}^{\hat{H}} := \{ \hat{\phi}^{\hat{H}}_{1},\ldots,\hat{\phi}^{\hat{H}}_{q} \} \subset L^2(\hat{H})$ be the optimal dictionary, and $\{ \ell^{(i) \hat{H}}_{j} \}_{1\leq i \leq n, 1 \leq j \leq q} \} \subset B(L^2(\hat{H}))$ be the optimal operators to \eqref{eq:GIDL_opt_FT_lagrangian}.  As discussed in Section \ref{sec:approx}, one can embed $\hat{\Phi}^{\hat{H}}$ into $L^2(\hat{G})$ in the natural way, and we let $\hat{\Phi}^{\hat{H}}$ denote such an embedding.  As discussed also in Section \ref{sec:approx}, we choose the embedding of $\{ \ell^{(i) \hat{H}}_{j} \}$ into $B(L^2(\hat{G}))$ to be such that $\|\ell^{\hat{H}} \|_{\hat{G}} = \| \ell^{\hat{H}} \|_{\hat{H}}$.


\begin{proof}[Proof of Theorem \ref{thm:objval_increasing} and \ref{thm:objval_decreasing}]
Let $\hat{\phi}^{\hat{H}}$ and $\hat{\ell}^{\hat{H}}$ denote the optimal dictionary elements and operators in \eqref{eq:GIDL_opt_FT_lagrangian}.  We consider what happens when we substitute these into the objective \eqref{eq:GIDL_opt_FT_lagrangian_full}.  


First, we have
\begin{equation*}
\begin{aligned}
& \sum_{\xi \in \hat{G}}  \frac{\mathrm{dim}(V_\xi)^2}{2} \| \hat{y}_{\xi} - \sum_{j=1}^{q} \hat{\phi}^{\hat{H}}_{j,\xi} (\ell^{\hat{H}}_{j,\xi})^* \|_F^2 \\
= ~~& \sum_{\xi \in \hat{H}}  \frac{\mathrm{dim}(V_\xi)^2}{2} \| \hat{y}_{\xi} - \sum_{j=1}^{q} \hat{\phi}^{\hat{H}}_{j,\xi} (\ell^{\hat{H}}_{j,\xi})^* \|_F^2 + \sum_{\xi \in \hat{G} \backslash \hat{H} }  \frac{\mathrm{dim}(V_\xi)^2}{2} \| \hat{y}_{\xi} \|_F^2.
\end{aligned}
\end{equation*}  

Second, we have $\| \ell_j^{\hat{H}} \|_{\hat{G}} = \| \ell_j^{\hat{H}} \|^+_{\hat{H}}$, as we noted earlier.  

Third, we have
\begin{equation*}
\sum_{\xi \in \hat{G}} \mathrm{dim}(V_\xi)^2 \| \hat{\phi}^{\hat{H}}_{\xi} \|_F^2 =
\sum_{\xi \in \hat{H}} \mathrm{dim}(V_\xi)^2 \| \hat{\phi}^{\hat{H}}_{\xi} \|_F^2
\end{equation*}
because $\hat{\phi}^{\hat{H}} \in L^2(\hat{H})$.

Let $\epsilon_{\hat{H}}^{(i)} = \sum_{\xi \in \hat{G} \backslash \hat{H} }  \frac{\mathrm{dim}(V_\xi)^2}{2} \| \hat{y}^{(i)}_{\xi} \|_F^2$.  By summing up the three components above, we have $\mathrm{OPT}-\hat{H} + \sum_i \epsilon_{\hat{H}}^{(i)} = \mathrm{OBJ}-\hat{H}$.  By optimality of $\hat{\Phi}^{\hat{G}}$, we have $\mathrm{OBJ}-\hat{H} \geq \mathrm{OPT}-\hat{G}$.  By Proposition \ref{thm:dictapprox_obj}, we have $\mathrm{OPT}-\hat{G} \geq \mathrm{OPT}-\hat{H}$.  Combining these, we have
\begin{equation*}
\mathrm{OPT}-\hat{H} + \sum_i \epsilon_{\hat{H}}^{(i)} = \mathrm{OBJ}-\hat{H} \geq \mathrm{OPT}-\hat{G} \geq \mathrm{OPT}-\hat{H}.
\end{equation*}
Since the data $y$ has finite $\ell_2$-norm, we have $\sum \epsilon_{\hat{H}_k}^{(i)} \rightarrow 0$ as $\hat{H}_k \uparrow \hat{G}$.  It follows that $\mathrm{OPT}-\hat{H}_k \uparrow \mathrm{OPT}-\hat{G}$ and $\mathrm{OBJ}-\hat{H}_k \downarrow \mathrm{OPT}-\hat{G}$.
\end{proof}

\begin{proof}[Proof of Proposition \ref{thm:dictapprox}]
By Theorem \ref{thm:objval_decreasing}, we have $\mathrm{OPT}-\hat{H}_k \downarrow \mathrm{OPT}-\hat{G}$.  In particular, it guarantees that $\mathrm{OPT}-\hat{H}_k \leq \mathrm{OPT}+\epsilon$ eventually.  Then by canonical uniqueness $D(\hat{\Phi},\hat{\Phi}^{\hat{H}_k}) \leq \delta(\epsilon)$.  Then take $\epsilon \rightarrow 0$ so that $\delta(\epsilon) \rightarrow 0$, from which the result follows.
\end{proof}

\section{Experimental Details of Section \ref{sec:so3cd}} \label{sec:so3cd_supp} In this section, we describe the details of our numerical experiment in Section \ref{sec:so3cd}.

\textbf{Data.}  Our raw dataset are handwritten digits from the MNIST dataset \cite{mnist}.  These are $28 \times 28$-pixel images of hand-written digits from $0$ to $9$ in grayscale with values ranging from $0$ (for black) to $1$ (for white).  We apply a sequence of pre-processing steps to convert these handwritten digits to functions over SO(3) to apply our procedure.  First, we translate these images so the cartesian coordinates are $[-\sqrt{2},\sqrt{2}] \times [-\sqrt{2},\sqrt{2}]$.  Second, we define a function $h : S^2 \rightarrow \mathbb{R}$ so that
\begin{equation*}
h( (x,y,z)^T ) := \text{Image Intensity at }(x/(1+z),y/(1+z)).
\end{equation*}
The intensity values of $h$ are initially specified only on a lattice grid -- these correspond to locations of the pixels of the original image.  For points off the lattice grid $(x,y)$ but within the domain $[-1/\sqrt{2},1/\sqrt{2}] \times [-1/\sqrt{2},1/\sqrt{2}]$, we define the value of $h$ via a bilinear interpolation using the four lattice points $(\lfloor x \rfloor, \lfloor y \rfloor), (\lceil x \rceil, \lfloor y \rfloor),  (\lfloor x \rfloor, \lceil y \rceil),  (\lceil x \rceil, \lceil y \rceil)$.  As for values of $x,y$ that fall outside the domain $[-1/\sqrt{2},1/\sqrt{2}] \times [-1/\sqrt{2},1/\sqrt{2}]$, we define the value of $h$ to be $0$.  Third, we define function $f :\text{SO}(3) \rightarrow \mathbb{R}$ by
\begin{equation*}
f ( \left( \begin{array}{c|c|c} u & v & w 
\end{array} \right) ) := h(u) \qquad \text{where} \qquad u,v,w \in S^2.
\end{equation*}
Here, we identify an element in SO(3) by a $3\times 3$-dimensional matrix whose columns are $u,v,w$.  Fourth, we compute the Fourier coefficients corresponding to $f$, which we denote by $\hat{y}$.  Recall $\hat{y}$ is a countable sequence of complex matrices of size $(2j+1) \times (2j+1)$, indexed by $j \in \{ 0,1,\ldots \}$.  In our implementation, we truncated the Fourier series to $j \in \{ 0,1,\ldots, 10\}$.  

\textbf{Numerical Experiment.}  We apply our dictionary learning procedure to learn a dictionary comprising one dictionary element $\hat{\phi}_{\rm opt}$ for the digits $\{0,1,\ldots,9\}$.  For each digit, our dataset comprises $200$ images.  We perform $10$ outerloop iterations in our dictionary learning procedure.

\textbf{Results.}  The computed dictionary elements are Fourier coefficients.  To visualize these as images, we need to reverse the pre-processing steps.  Let $\hat{\phi}_{\rm opt}$ be a dictionary element.  First, we perform the inverse Fourier transform to obtain $\phi_{\rm opt} \in L^2(\mathrm{SO}(3))$.  Second, define the spherical image $h_{\rm opt}$ corresponding to $\phi_{\rm opt}$ by
\begin{equation*}
h_{\rm opt} (u) = \int_{R : R = (u|*|*)} \phi_{\rm opt} (R) \mu(dR).
\end{equation*}
Here, $\mu$ is the uniform measure.  To approximate this integral, notice that -- assuming $u$ is at the north-pole -- the remaining columns lie on the equator.  We set the second argument -- say $v$ -- to be equally spaced points on the equator, and we set the third argument -- say $w$ -- to be $w = \pm (u \times v)$.  Finally, the pixel intensity of the learned dictionary at location $x,y$, $-1/\sqrt{2} \leq x,y \leq 1/\sqrt{2}$, is given by $\hat{h}( (x,y, \sqrt{1 - x^2 - y^2})^T )$.

In Figure \ref{fig:dictionaryelements}, we show the learned dictionary elements for the digits $0$ to $9$.  For digits $0$, $1$ and $5$, we note some resemblance between the learned dictionary element and the actual digit; for the other digits, the resemblance is less clear.  (We also note that the dictionary elements for the digits $6$ and $9$ bear similarities not observed in other digits.)  This can be explained by a couple of reasons: For instance digits $0$ and $1$ are structurally quite `simple' in that $0$ is a single circle while $1$ is a single stroke, and that makes it easier for the dictionary learning algorithm to identify such structures; the remaining digits are more complex.  Another reason is that the Euclidean loss, while amenable to computation, is not the ideal loss function for image processing -- alternatives such as the Wasserstein distances are better suited for images, though they lead to algorithms that are computationally far more expensive.  We note in passing that the objects in Figure \ref{fig:dictionaryelements} are dictionary elements, whose linear combinations would be relevant for digits in the MNIST problem.  It is not apparently evident if the dictionary elements themselves should possess any interpretable structure in the context of the actual digits. 

\textbf{Computational considerations.}  We make some brief remarks concerning the computational aspects of our experiment.  A consequence of the function $f$ being real-valued is that the $j$-th matrix coefficient $\hat{y} := \hat{y}_{\xi = j} \in \mathbb{C}^{(2j+1) \times (2j+1)}$ (here, the index runs between $-j$ and $j$) satisfies the relation
\begin{equation*}
X_{-m,-m'} = (-1)^j X_{-m,m'} = (-1)^{j + m + m'} X_{m,-m'}^* = (-1)^{m+m'} X_{m,m'}.
\end{equation*}
The degrees of freedom within the $j$-th matrix is $(2j+1)(j+1)$ for $j$ even and $2j^2+j$ for $j$ odd.  Hence, in truncating the data to Fourier coefficients in $j \in \{ 0,1,\ldots, 10\}$, the ambient (real) dimension in our problem is $891$, which is comparable to datasets in prior works.  Our earlier proposal -- that is, the relaxations described in Section \ref{sec:sdp_atomicnorm_multivariate} -- requires solving a SDP over matrices in $21^3 \times 21^3$ dimensions \emph{per} data-point.  This is near the limit of what general purpose SDP solvers are able to solve in a single instance, much less repeated over the dataset and multiple interations.

\bibliography{bib_repdl, Classfn}

\end{document}